\documentclass[envcountsect,envcountsame,runningheads]{llncs}
\usepackage{graphicx}
\usepackage[hidelinks]{hyperref}

\makeatletter
\AtBeginDocument{%
	\def\doi#1{\url{https://doi.org/#1}}}
\makeatother

\usepackage[dvipsnames]{xcolor}
\usepackage{amsfonts}
\usepackage{amsmath}
\usepackage[classfont=bold, langfont=roman]{complexity}
\usepackage{xspace}
\usepackage{tikz}
\usepackage[%
	capitalize,
	noabbrev,
	english,
]{cleveref}
\usepackage{thm-restate}
\usepackage{enumitem}
\usepackage{fontawesome5}

\colorlet{markerColor1}{YellowGreen}
\colorlet{markerColor2}{Orange}

\colorlet{contourLineColor}{Gray!70}
\colorlet{pathColor1}{Red}
\colorlet{pathMarkerColor1}{Red!20}
\colorlet{pathColor2}{ProcessBlue!80!black}
\colorlet{pathMarkerColor2}{ProcessBlue!20}
\colorlet{pathColor3}{green!90!black}
\colorlet{pathMarkerColor3}{green!20}
\colorlet{pathColor4}{Orange!70!Yellow}
\colorlet{pathMarkerColor4}{Orange!20}
\newcommand{\NN}{\mathbb{N}}

\newcommand{\cA}{\mathcal{A}}
\newcommand{\cC}{\mathcal{C}}
\newcommand{\cI}{\mathcal{I}}
\newcommand{\cQ}{\mathcal{Q}}

\newcommand{\arctail}[1]{\alpha(#1)}
\newcommand{\archead}[1]{\omega(#1)}

\newcommand{\inarcs}[2][]{\delta^{\operatorname{in}}_{#1}(#2)}
\newcommand{\outarcs}[2][]{\delta^{\operatorname{out}}_{#1}(#2)}
\newcommand{\indegree}[2][]{\operatorname{deg}^{\operatorname{in}}_{#1}({#2})}
\newcommand{\outdegree}[2][]{\operatorname{deg}^{\operatorname{out}}_{#1}({#2})}
\newcommand{\degree}[2][1]{\operatorname{deg}_{#1}({#2})}

\newcommand{\setofstpaths}{\mathcal{P}}

\newcommand{\incidentedges}[2][]{\delta_{#1}(#2)}

\newcommand{\gadget}[1]{\operatorname{gad}(#1)}

\newclass{\SigmaTwoP}{\ensuremath{\Sigma_2\P}}

\let\OrigSigmaTwoP\SigmaTwoP
\renewcommand{\SigmaTwoP}{\OrigSigmaTwoP\xspace}
\let\OrigNP\NP
\renewcommand{\NP}{\OrigNP\xspace}

\newcommand{\ADP}{\lang{ADP}\xspace}
\newcommand{\PAFP}{\lang{PAFP}\xspace}
\newcommand{\QSAT}[1][2]{\lang{QSAT_{#1}}\xspace}
\newcommand{\SFP}{\lang{SFP}\xspace}
\newcommand{\SigmaTwoSat}{\lang{\Sigma_2 SAT}\xspace}

\newcommand{\bigOt}[1]{\mathcal{O}(#1)}
\newcommand{\bigO}[1]{\mathcal{O}\left(#1\right)}

\newcommand{\DNF}[1][3]{#1-DNF\xspace}

\newcommand{\ownchoose}[3][1]{\genfrac(){0pt}{#1}{#2}{#3}}

\newcommand{\restrict}[2]{\left.#1\right|_{#2}}
\newcommand{\setcolon}{:}  
\newcommand{\true}{\texttt{true}\xspace}
\newcommand{\false}{\texttt{false}\xspace}\usetikzlibrary{automata}
\usetikzlibrary{backgrounds}
\usetikzlibrary{calc}
\usetikzlibrary{math}
\usetikzlibrary{shapes.geometric}
\usetikzlibrary{decorations.pathreplacing}
\usetikzlibrary{decorations.pathmorphing}
\usetikzlibrary{arrows.meta}

\tikzset
{
	>=latex,
	stdArc/.style = {->},
	snakeArc/.style = {->, decorate, decoration={snake,amplitude=.5mm,post length=1mm}},
	stdEdge/.style = {-},
	stdVertex/.style = {circle, draw, inner sep=2pt, minimum size=8pt},
	stdVertexPhantom/.style = {circle, inner sep=2pt, minimum size=8pt},
	smallVertex/.style = {circle, draw, inner sep=0pt, minimum size=5pt},
	smallVertexPhantom/.style = {circle, inner sep=0pt, minimum size=5pt},
	arrowVertex/.style = {dart, fill, shape border uses incircle, inner sep=1.4pt, draw},
	smallArrowVertex/.style = {dart, fill, shape border uses incircle, inner sep=1pt, draw},
	marker1/.style = {markerColor1,line width=10pt,opacity=.7},
	marker2/.style = {markerColor2,line width=10pt,opacity=.7},
	pathMarker1/.style = {pathMarkerColor1, line width=8pt},
	pathMarker2/.style = {pathMarkerColor2, line width=8pt},
	pathMarker3/.style = {pathMarkerColor3, line width=8pt},
	pathMarker3/.style = {pathMarkerColor4, line width=8pt},
	contourLine/.style = {color=contourLineColor},
}

\makeatletter
\@spothm{assumption}[theorem]{Assumption}{\bfseries}{\itshape}
\@spothm{notation}[theorem]{Notation}{\bfseries}{\itshape}
\if@cref@capitalise
	\crefname{assumption}{Assumption}{Assumptions}
	\crefname{notation}{Notation}{Notations}
\else
	\crefname{assumption}{assumption}{assumptions}
	\crefname{notation}{notation}{notations}
\fi
\makeatother

\definecolor{orcidcol}{HTML}{A6CE39}
\renewcommand\orcidID[1]{\,\href{https://orcid.org/#1}{\color{orcidcol}\faOrcid}}

\begin{document}
\title{Almost Disjoint Paths and Separating by Forbidden Pairs} 
%
%
\author{Oliver Bachtler\orcidID{0000-0001-7942-0750} \and
Tim Bergner\orcidID{0000-0002-7899-6835} \and
Sven O. Krumke\orcidID{0000-0002-8726-9963}}
\authorrunning{O.\ Bachtler, T.\ Bergner, and S.\ O.\ Krumke}
%
\institute{Department of Mathematics, Technische~Universität~Kaiserslautern, Paul-Ehrlich-Str.~14, 67663~Kaiserslautern, Germany\\
\email{\{bachtler,bergner,krumke\}@mathematik.uni-kl.de}}
\maketitle              
\begin{abstract}
By Menger's theorem, the maximum number of arc-disjoint paths from a vertex~$s$ to a vertex~$t$ in a directed graph equals the minimum number of arcs needed to disconnect $s$ and~$t$, i.e., the minimum size of an $s$-$t$-cut.
The max-flow problem in a network with unit capacities is equivalent to the arc-disjoint paths problem.
Moreover, the max-flow and min-cut problems form a strongly dual pair.
We relax the disjointedness requirement on the paths, allowing them to be almost disjoint, meaning they may share up to one arc.
The resulting almost disjoint paths problem (\ADP) asks for $k$ $s$-$t$-paths such that any two of them are almost disjoint.
The separating by forbidden pairs problem (\SFP) is the corresponding dual problem and calls for a set of $k$ arc pairs such that every $s$-$t$-path contains both arcs of at least one such pair.

In this paper, we explore these two problems, showing that they have an unbounded duality gap in general and analyzing their complexity.
We prove that \ADP is \NP-complete when $k$ is part of the input and that \SFP is \SigmaTwoP-complete, even for acyclic graphs.
Furthermore, we efficiently solve \ADP when $k \leq 2$ is fixed and present a polynomial time algorithm based on dynamic programming for \ADP when $k$ is constant and the considered graphs are acyclic.
\end{abstract}
\section{Introduction} \label{sec:intro}
In many applications, customers receive various offers from which they can select one or between which they can switch.
Usually, these offers should be as diverse as possible to provide the customer with many different options.
A common use case is the construction of alternative routes in transportation or road networks.
These make sense in this context, for example to avoid route closures, heavy traffic, or tolls.
Another example where alternative routes are of use is to distribute risk.
For example, if dangerous goods need to be transported regularly, alternative routes that affect different people allow for an equal risk distribution amongst the people exposed.
Several practical algorithms for computing alternative routes have been developed, see, for example, \cite{ADGW10,AEB00,BDGS11,DGS05,JKPK09}.

On the graph-theoretic side, the (arc- or vertex-) disjoint paths problem is well-studied.
Determining a maximum number of disjoint $s$-$t$-paths can easily be done using maximum flow techniques~\cite{AMO93}.
By Menger's theorem~\cite{Men27,Die10} this number is equal to the minimum number of arcs needed to separate $s$ from $t$.
This result follows from the max-flow min-cut theorem~\cite{DF55}, which shows that these two problems form a strongly dual pair.

The following extension of the disjoint paths problem is also well-understood:
given $k$ pairs of terminals $(s_1,t_1),\ldots,(s_k, t_k)$, the objective is to find disjoint $s_i$-$t_i$-paths.
For undirected graphs it is solvable in polynomial time if $k$ is constant (see \cite{RS95} for a cubic and \cite{KKR12} for a quadratic algorithm) and \NP-complete in general~\cite{EIS76}.
In the case of directed graphs, a single path is easy and two paths are already \NP-complete~\cite{FHW80}.
The problem remains \NP-complete for few paths even on very restricted graph classes like acyclic, Eulerian, or planar graphs~\cite{Vyg95}.

Another possible extension is to ask for $k$ disjoint $s$-$t$-paths that are short, which again makes sense for routing purposes.
Suurballe~\cite{Suu74} describes an algorithm for this problem that is based on shortest path labelings.
It is possible to combine both extensions and ask for shortest paths between different terminals.
Eilam-Tzoreff~\cite{Eil98} shows that these problems in all configurations (for directed and undirected graphs with vertex- or arc-disjoint paths) are \NP-complete and also provides a polynomial algorithm for two paths in an undirected graph with positive edge-weights.
Bérczi and Kobayashi~\cite{BK17} present a polynomial algorithm for the directed version, also with two paths and positive arc-lengths.

In contrast, the same problem where the paths need not be completely disjoint has not garnered as much attention in the literature.
There are several natural relaxations of the disjointedness condition:
a first option allows arcs to be part of more than one path, say each arc may be used by two.
This problem can be solved in the same way as the disjoint paths problem, by a maximum flow computation in the graph with arcs of capacity~2.
An alternative is to allow some arcs, say one, to be part of an arbitrary number of paths.
This, too, can be solved by maximum flow techniques, but requires one flow computation for each arc.
These flows are computed in the graphs where one arc's capacity is set to infinity (and the rest remain at capacity~1).
The choice we cover in this paper allows paths to have at most one arc in common and leads to the almost disjoint paths problem (\ADP).
More precisely, we wish to find $k$ paths in a directed graph that are almost disjoint, meaning that any two of them share at most one arc.
This could be generalized to allow for some fixed number of common arcs.

Most of the literature on nearly disjoint paths is of a very practical nature as is evidenced by the initial examples we presented. 
We now discuss some of the (rarer) theoretical results that exist for problems similar to \ADP.
Liu et al.~\cite{LJYZ18} introduce the $k$ shortest paths with diversity problem, in which the goal is to find a set of sufficiently dissimilar paths of maximum size (bounded by~$k$).
Of such sets, the one that minimizes the total path length is optimal.
For this problem, an (incorrect) \NP-hardness proof as well as a greedy framework is presented.
Chondrogiannis et al.~\cite{CBDLB18} fix said \NP-hardness proof, showing that the problem is indeed strongly \NP-hard, and develop an exact algorithm for it as well as heuristics.
Moreover, Chondrogiannis et al.\ \cite{CBG20} consider the problem of finding $k$~shortest paths with limited overlap.
They prove that this variant is weakly \NP-hard and develop two exact algorithms for it. 
The problems here are similar to \ADP in the sense that they look for paths that are sufficiently dissimilar, though the measures used always result in similarity values between~0 and~1 because they compare the number of arcs in common with some function based on the lengths of the two paths.
Additionally, they want to minimize the total length of the paths found.

Inspired by the strong duality of max-flow and min-cut, we make analogous considerations for our almost disjoint paths problem.
By dualizing the linear relaxation of an integer programming formulation, we obtain a dual problem which we call separating by forbidden pairs (\SFP).
Its goal is to select as few arc pairs as possible such that every $s$-$t$-path in~$G$ contains both arcs of at least one chosen pair.
While the linear programming relaxations of these problems form a dual pair and thus have the same objective value~\cite{GKT51}, the corresponding integer versions are only weakly dual.
Note that in the min-cut problem we select an arc on every $s$-$t$-path whereas in \SFP we select a pair of arcs on every $s$-$t$-path.

Apart from being dual to \ADP, \SFP adds another level on top of the well-known path avoiding forbidden pairs problem (\PAFP).
In the latter problem one is given a set of arc pairs~$\cA$ and has to identify whether some $s$-$t$-path avoids all pairs in~$\cA$ (meaning \SFP asks for a set that makes the corresponding \PAFP instance unsolvable).
Originating from the field of automated software testing~\cite{KSG73}, \PAFP also has applications in aircraft routing~\cite{BBB+15} and biology, for example in peptide sequencing~\cite{CKT+00} or predicting gene structures~\cite{KVB09}.
The \PAFP is \NP-complete~\cite{GMO76} and various restrictions on the set of forbidden pairs have been considered.
The problem becomes solvable if the pairs satisfy certain symmetry properties~\cite{Yin97} or if they have a hierarchical structure~\cite{KP09} while it remains \NP-hard even if the pairs have a halving structure~\cite{KP09} or no two pairs are nested~\cite{Kov13}.
The structure of the \PAFP polytope has been analyzed \cite{BBB+15} and Hajiaghayi et al.~\cite{HKKM10} show that determining a path that uses a minimal number of forbidden pairs cannot have a sublinear approximation algorithm.

Note that some of the referenced papers consider forbidden pairs of vertices instead of pairs of arcs.
However, these two variants can be converted into one another by standard constructions.
Moreover, as this problem is usually regarded on acyclic graphs, we also specifically regard \ADP and \SFP under this restriction.

\paragraph{Contributions.}
We analyze the complexity of the weakly dual problems \ADP and \SFP.
In \cref{sec:preliminiaries} we define these problems in addition to introducing notation and stating general assumptions.
We also investigate their duality, showing that, unlike for the max-flow and min-cut problems, the duality gap here can be arbitrarily large (even for acyclic graphs).
However, for graphs that have an $s$-$t$-cut of capacity~1, they are actually strongly dual, that is, the duality gap is~0.
\Cref{sec:adp} then covers \ADP, starting \cref{sec:adp:dynamic_program} with the case that $k \leq 2$ which can be solved efficiently using flow techniques.
By restricting to acyclic graphs, the case where $k$ is constant also becomes tractable.
We then use the resulting algorithm to solve the problem for general graphs.
\begin{restatable}{theorem}{adppoly} \label{thm:adppoly}
	For constant~$k$, \ADP is polynomial time solvable.
\end{restatable}
The case where $k$ is part of the input is discussed in \cref{sec:adp:completeness}, resulting in a proof of the next theorem.
\begin{restatable}{theorem}{adpnp} \label{thm:adpnp}
	\label{thm:adp_np-complete}
	\ADP is \NP-complete, even on acyclic graphs.
\end{restatable}
We also note that this result remains true when paths are allowed to have a fixed number of arcs in common, instead of just one.
Finally, in \cref{sec:sfp}, we consider \SFP and prove $\SigmaTwoP$-completeness.
\begin{restatable}{theorem}{sfpstwop} \label{thm:sfpstwop}
	\label{thm:sfp_sigma2p-complete}
	\SFP is \SigmaTwoP-complete, even on acyclic graphs.
\end{restatable}

\section{Preliminaries} \label{sec:preliminiaries}
We begin this section by introducing our basic notation.
Next, we define the problems we consider, that is, the almost disjoint paths problem (\ADP) and the separating by forbidden pairs problem (\SFP).
In this context, we also establish a connection to the path avoiding forbidden pairs problem (\PAFP).
Afterwards, we investigate the duality of \ADP and \SFP, showing that they are only weakly dual with an arbitrarily large duality gap in general, but we also present a special case where they are strongly dual.
We complete this section by making some non-restrictive assumptions on the graphs we consider.

\subsection{Basic Notation}

Throughout this paper, directed graphs will be denoted by $G = (V, A)$ and have a designated source~$s \in V$ and target~$t \in V$.
When a graph is fixed, the set of all its $s$-$t$-paths is denoted by~$\setofstpaths$, and we denote the start-vertex of an arc~$a \in A$ by~$\arctail{a}$ and its end-vertex by~$\archead{a}$.
By $\indegree[G]{v}$ and $\outdegree[G]{v}$ we denote in-\ and outdegree of a vertex~$v$ in the graph~$G$, and $\inarcs[G]{v}$ and $\outarcs[G]{v}$ denote its incoming and outgoing arcs, respectively.
For the subgraph induced by a subset $U \subseteq V$ we write~$G[U]$.
Given a set~$S$ and a natural number~$k \in \NN$ we denote the set of all $k$-element subsets of~$S$ by~$\ownchoose{S}{k}$.

\subsection{Problems}

Let us now define the two problems we are interested in, starting with the almost disjoint paths problem, in which paths are allowed to share up to one arc, in contrast to the strict disjoint paths problem, in which they may not.
\begin{definition} \label{def:almost_disjoint}
	A set of paths in a (directed) graph is called \emph{almost disjoint} if every two paths of this set have at most one arc in common.
\end{definition}

\begin{problem} \label{def:adp}
	The \emph{Almost Disjoint Paths problem} (\ADP) is given by a directed graph~$G = (V, A)$ together with two designated vertices $s, t \in V$ and a natural number $k \in \NN$.
	The question is whether $k$ almost disjoint $s$-$t$-paths exist.
\end{problem}

Its weak dual (as we will see in the next subsection) is the separating by forbidden pairs problem:
\begin{problem} \label{def:sfp}
	The \emph{Separating by Forbidden Pairs problem} (\SFP) is given by a directed graph $G = (V, A)$, two designated vertices $s, t \in V$, and a natural number $k \in \NN$.
	The question is whether there exist $k$~(unordered) pairs of arcs, that is, $\cA \subseteq \ownchoose{A}{2}$ with $|\cA| = k$,  that separate ($s$ and $t$ in) $G$, meaning that every $s$-$t$-path contains both arcs of at least one pair $q \in \cA$.
\end{problem}
We note that \SFP contains the path avoiding forbidden pairs problem (\PAFP) as a subproblem.
In \PAFP, a directed graph $G$ and a set of arc pairs $\cA$ is given.
The goal is to find an $s$-$t$-path in~$G$ that does not contain both arcs of a pair in $\cA$.
In this light, \SFP asks for a set $\cA$ such that there is no solution to the corresponding \PAFP instance.

\subsection{Duality}

The problems \ADP and \SFP form a pair of weakly dual problems as the following integer programming formulations for their optimization variants demonstrate.
\begin{align*}
	\max & \sum_{P \in \setofstpaths} y_P && & \min & \sum_{q \in \ownchoose{A}{2}} x_q \\
	\text{s.t.} ~ & \sum_{\substack{P \in \setofstpaths\\ q \subseteq A(P)}} y_P \leq 1 && \text{f.a. } q \in \ownchoose{A}{2} & \text{s.t.} ~ & \sum_{q \in \ownchoose{A(P)}{2}} x_q \geq 1 && \text{f.a. } P \in \setofstpaths \\
	& y_P \in \{0, 1\} && \text{f.a. } P \in \setofstpaths & & x_q \in \{0, 1\} && \text{f.a. } q \in \ownchoose{A}{2}
\end{align*}
More precisely, the LP-relaxations of these formulations that we obtain by replacing $y_P \in \{0, 1\}$ and $x_q \in \{0, 1\}$ by $y_P \geq 0$ and $x_q \geq 0$ form a dual pair.
In contrast to Menger's theorem~\cite{Men27,Die10}, almost disjoint
paths have an unbounded duality gap as the following lemma shows:

\begin{lemma}
	The duality gap between \ADP and \SFP is unbounded.
\end{lemma}
\begin{proof}
	Let $P^k_{\ell}$ denote an $s$-$t$-path of length~$\ell$ whose $\ell$~arcs are all replaced by bunches of~$k$ parallel arcs.
	Assume $k \geq 2$ and $\ell \geq k + 1 + k (k - 1) / 2$.
	We prove that we need $k^2$ pairs to separate $s$ and~$t$ in $P^k_{\ell}$ but that we can find at most $k$ almost disjoint $s$-$t$-paths therein.

	We start by proving that we need $k^2$ forbidden pairs to separate $s$ and~$t$ in~$P^k_{\ell}$.
	The graph $P^k_{\ell}$ has $k^{\ell}$ different $s$-$t$-paths.
	For a fixed pair $q \in \ownchoose{A}{2}$ we either have 0 or~$k^{\ell - 2}$ $s$-$t$-paths using this pair (depending on whether both arcs are contained in the same bunch or not).
	Therefore, $y \equiv 1 / (k^{\ell - 2})$ is feasible for \ADP's LP-relaxation and has an objective value of $k^2$.
	That we need $k^2$ pairs to separate $s$ and~$t$ in~$P^k_{\ell}$ now follows from the strong duality theorem of linear programming \cite{GKT51}.

	To show that $P^k_{\ell}$ admits only $k$~almost disjoint paths let $\cQ$ be a maximum set of almost disjoint $s$-$t$-paths and enumerate them arbitrarily.
	Since the graph has $k$~disjoint $s$-$t$-paths, we have $|\cQ| \geq k$.
	The second path has at most one arc in common with the first path and, if this is the case, we can assume without loss of generality that it is contained in the first bunch.
	More general, the $i$-th path has at most one arc in common with any of the first $i - 1$~paths and we can assume (again without loss of generality) that these arcs are contained in the first $\sum_{j=1}^{i-1} j$ bunches.

	Hence, if two of the first $k$~paths in~$\cQ$ have an arc in common, we can assume that it is in the first $\sum_{j=1}^{k-1} j = k (k - 1) / 2$ bunches.
	Our assumption $\ell \geq k + 1 + k (k - 1) / 2$ thus implies that each arc from the last $k + 1$ bunches is contained in at most one of the first $k$ paths.
	Each bunch consists of $k$ arcs.
	Thus, each arc of the last $k + 1$ bunches is contained in exactly one of the first $k$ paths.
	Another path now has to use one arc from each of the $k + 1$ last bunches.
	However, this means that it has at least two arcs in common with one of the first $k$ paths.
	Hence, there is no further path and it follows $|\cQ| = k$.
	\qed
\end{proof}

However, the duality gap is not always unbounded.
For example, if we restrict ourselves to graphs that have an $s$-$t$-cut with a single outgoing arc, the duality gap disappears.
\begin{lemma}
	Let $G$ be a directed graph and $s,t\in V$.
	If $G$ has an $s$-$t$-cut~$(S, T)$ with a single outgoing arc $\outarcs{S} = \{uv\}$, the duality gap is zero and we can solve \ADP and \SFP in polynomial time.
\end{lemma}
\begin{proof}
	Note that every $s$-$t$-path in $G$ must use the arc $uv$, so the paths in a set of almost disjoint $s$-$t$-paths must be disjoint aside from $uv$.
	We can compute maximum sets $\setofstpaths_s$ and~$\setofstpaths_t$ of disjoint $s$-$u$- and $v$-$t$-paths, respectively.
	Combining $k = \min\{|\setofstpaths_s|, |\setofstpaths_t|\}$ paths from each of these sets results in $k$~almost disjoint $s$-$t$-paths.
	In addition, one of the subgraphs $G[S]$ or~$G[T]$ has a cut with $k$~outgoing arcs.
	Bundling $uv$ with each of these arcs results in $k$ forbidden pairs separating $s$ and~$t$.
	\qed
\end{proof}

\subsection{Assumptions}

If the direct arc~$st$ is contained in the graph, this arc itself forms an $s$-$t$-path.
With respect to \ADP this path is disjoint from every other $s$-$t$-path and can always be added to a set of almost disjoint paths.
Regarding \SFP this is a path that never contains a forbidden pair as it only has length~1.
Thus, every instance containing the direct arc~$st$ cannot be separated.
This justifies the following assumption.
\begin{assumption} \label{ass:no-direct-arc-st}
	The direct arc~$st \notin A$ is not contained in the graph.
\end{assumption}

For both problems, \ADP as well as \SFP, every arc and every vertex that is not contained in any $s$-$t$-path is irrelevant and can be removed.
Hence, we assume such arcs and vertices do not exist.
\begin{assumption} \label{ass:every-arc-in-st-path} \label{ass:every-vertex-in-st-path}
	Every arc and every vertex of the graph is contained in an $s$-$t$-path.
\end{assumption}

\section{Almost Disjoint Paths} \label{sec:adp}
In this section, we inspect the almost disjoint paths problem.
We first show how to solve it for $k\leq 2$ before presenting a dynamic program that deals with any constant $k$.
Lastly, we prove that the problem is \NP-complete when $k$ is part of the input, even if the graph is acyclic.

\subsection{Constantly Many Paths} \label{sec:adp:dynamic_program}

For $k = 1$, \ADP reduces to reachability, which can be solved in linear time.
We can check whether two almost disjoint paths exist by computing one maximum flow per arc.
For $a \in A$ we can define arc capacities~$c_a$ with $c_a(a) = 2$ and $c_a(a') = 1$ for $a' \in A \setminus \{a\}$.
An $s$-$t$-flow with respect to $c_a$ of value~$\ell$ corresponds to $\ell$ many $s$-$t$-paths that have at most the arc~$a$ in common.
Therefore, two almost disjoint paths exist if and only if an $s$-$t$-flow of value at least~2 exists with respect to arc capacities~$c_a$ for some arc~$a$.
Hence, we can solve the problem for $k=2$ by computing a maximum $s$-$t$-flow with respect to $c_a$ for each $a\in A$ and checking whether one of them has value at least~2.
In fact, instead of computing a maximum flow for each capacity $c_a$ we can check whether a flow of value at least~2 exists by simply making (at most) two flow augmentations, which only require linear time.
\begin{lemma}
	For $k=2$, \ADP can be solved in $\bigO{|A|(|V|+|A|)}$ time.
\end{lemma}
However, this technique does not generalize to $k > 2$ because in this case several arcs might be contained in multiple paths and we cannot guarantee that two paths only share a single arc.
Instead, we use a dynamic program to find a constant number~$k$ of almost disjoint $s$-$t$-paths in polynomial time.
We first derive the dynamic program for acyclic graphs and then adapt it to graphs that might contain cycles.

\begin{theorem} \label{thm:adp_poly_acyclic_graphs}
	For constant $k$, \ADP is polynomial time solvable on acyclic graphs.
\end{theorem}
\begin{proof}
	Let $m = |A|$ be the number of arcs of a graph~$G$ and assume~$k$ to be fixed within this subsection.
	By assumption, $G$ is acyclic and therefore admits a topological ordering $\nu \colon V \to \NN$ of its vertices.
	\Cref{ass:every-vertex-in-st-path} implies that the source~$s$ (the target~$t$) always has the smallest (largest) value of~$\nu$.

	\subsubsection{States.}
	The dynamic program is based on states.
	A state~$((a_1, \dots, a_k), \cI)$ consists of $k$ (not necessarily disjoint) arcs and an intersection pattern~$\cI \subseteq \ownchoose{\{1, \dots, k\}}{2}$.
	We associate a state with a Boolean value~$x((a_1, \dots, a_k), \cI)$ that is \true if and only if $k$~almost disjoint paths $P_1, \dots, P_k$ with the following properties exist:
	\begin{itemize}
		\item For every $i \in \{1, \dots, k\}$ the path~$P_i$ is an $s$-$\archead{a_i}$-path whose last arc is~$a_i$.
		\item For $i \neq j$ the paths $P_i$ and~$P_j$ have an arc in common if and only if $\{i, j\} \in \cI$.
	\end{itemize}
	There are $m^k$~different possibilities to choose $k$ out of~$m$~arcs (with replacement).
	Additionally, we have $\bigOt{2^{k^2}}$~different intersection patterns, yielding $\bigOt{m^k 2^{k^2}}$~states in total.
	Note that this number is polynomial since we assume $k$ to be constant.

	\subsubsection{Comparing States.}
	To enable the computation of the truth values of all states with a dynamic program we have to order them appropriately.
	For this purpose we introduce the relation~$\prec$ on the states.
	Using the topological ordering~$\nu$ we define that $((a_1, \dots, a_k), \cI) \prec ((a'_1, \dots, a'_k), \cI')$ applies if and only if
	\begin{align*}
	\nu(\arctail{a_i}) \leq \nu(\arctail{a'_i}) &\text{ for all } i \in \{1, \dots, k\} \text{ and } \\
	\nu(\arctail{a_i}) < \nu(\arctail{a'_i}) &\text{ for at least one } i \in \{1, \dots, k\}.
	\end{align*}
	That is, we ignore the intersection pattern and compare the values in the topological ordering of the arc's start-vertices for each of the $k$~components separately.

	\subsubsection{Goal.}
	If we know the truth values of all states, we can determine whether $k$~almost disjoint $s$-$t$-paths in~$G$ exist.
	We only have to check whether a state with value \true exists whose arcs all enter the target~$t$.
	In fact, we can check this during the dynamic program when computing the truth values of the appropriate states.

	\subsubsection{Base.}
	Similarly, we proceed at the start by determining the truth values of all states whose arcs all leave the source~$s$.
	For any $k$~arcs $a_1, \dots, a_k \in \outarcs{s}$ and an arbitrary intersection pattern~$\cI$ we have that the value $x((a_1, \dots, a_k), \cI)$ is \true if and only if $\cI = \{ \{i, j\} \setcolon a_i = a_j \text{ for } i, j \in \{1, \dots, k\}, i \neq j\}$.

	\subsubsection{Recursion.}
	The dynamic program is based on a recursion that allows the computation of the truth value of a state based on the truth values of smaller states (with respect to~$\prec$).
	To derive this recursion let $((a_1, \dots, a_k), \cI)$ be a state.
	If all arcs $a_1, \dots, a_k$ start in~$s$, we are in the base case, which is already handled in the previous paragraph.
	Otherwise, an arc $a \in \{a_1, \dots, a_k\}$ maximizing the value $\nu(\arctail{a})$ satisfies $\arctail{a} \neq s$.
	Without loss of generality, we assume~$a = a_1$.

	If $a_1 = a_i$ but $\{1, i\} \notin \cI$ for some $i \in \{2, \dots, k\}$, the state must have truth value \false as any paths~$P_1$ and~$P_i$ ending with the arc~$a_1 = a_i$ have this arc in common.
	In the following, we therefore assume $\cC = \{\{1, i\} \colon a_1 = a_i, i \neq 1\} \subseteq \cI$ and define $\tilde{\cI} = \cI \setminus \cC$.
	We claim that the truth value of $((a_1, \dots, a_k), \cI)$ is the disjunction
	\begin{align} \label{eq:adp:dynamic_program:recursion}
	x((a_1, \dots, a_k), \cI) = \bigvee \left\{ x((\tilde{a}, a_2, \dots, a_k), \tilde{\cI}) \setcolon \tilde{a} \in \inarcs{\arctail{a_1}}\right\}
	\end{align}
	of truth values of smaller states.
	Note that all states in the disjunction are indeed smaller as $\nu(\arctail{\tilde{a}}) < \nu(\archead{\tilde{a}}) = \nu(\arctail{a_1})$ due to fact that $\nu$ is a topological ordering.
	We now prove the correctness.

	First, suppose that $x((a_1, \dots, a_k), \cI)$ is \true.
	Thus, there exist almost disjoint paths $P_1, \dots, P_k$ with intersection pattern~$\cI$ and last arcs $a_1, \dots, a_k$.
	The paths remain almost disjoint when removing the last arc~$a_1$ from~$P_1$.
	This removal changes the intersection pattern from $\cI$ to~$\tilde{\cI}$.
	This shows that $x((\tilde{a}, a_2, \dots, a_k), \tilde{\cI})$ is \true for $\tilde{a}$ being the penultimate arc on the path~$P_1$.
	As $\tilde{a} \in \inarcs{\arctail{a_1}}$, this value is contained in the disjunction from \cref{eq:adp:dynamic_program:recursion}.

	Now, suppose that $x((\tilde{a}, a_2, \dots, a_k), \tilde{\cI})$ is \true for an arc $\tilde{a} \in \inarcs{\arctail{a_1}}$ and let $\tilde{\cI}$ be as defined above.
	Again, there are almost disjoint paths $P_1, P_2, \dots, P_k$ with last arcs $\tilde{a}, a_2, \dots, a_k$ and intersection pattern~$\tilde{\cI}$.
	By the choice of the arc~$a_1$, we obtain that $a_1$ is the last arc of a path~$P_i$ whenever it is contained in~$P_i$.
	To see this, remember that $a_1$ is among the arcs~$\{a_1, \dots, a_k\}$ one whose start-vertex has the largest value~$\nu(\arctail{a_1})$ in the topological ordering.
	If $a_1$ is contained in a path~$P_i$ but not the last arc $a_1 \neq a_i$, we have $\nu(\arctail{a_i}) > \nu(\arctail{a_1})$ which is a contradiction.

	Since $\tilde{\cI} \cap \cC = \emptyset$, the path~$P_1$ shares no arc with another path that ends with~$a_1$.
	And because $a_1$ can only be the last arc of a path~$P_i$, we can extend~$P_1$ by the arc~$a_1$ while maintaining that the paths are almost disjoint.
	Moreover, the new intersection pattern is~$\cI$ as
	\begin{align*}
		\cI
		= \tilde{\cI} \cup \cC
		= \tilde{\cI} \cup \{\{1, i\} \colon a_1 = a_i, i \neq 1\}
		= \tilde{\cI} \cup \{\{1, i\} \setcolon a_1 \in P_i, i \neq 1\}.
	\end{align*}
	This shows the correctness of the recursion from \cref{eq:adp:dynamic_program:recursion}.
	Hence, we can compute the truth values of all states in polynomial time.
	\qed
\end{proof}
\adppoly*
\begin{proof}
	We prove the claim by converting~$G = (V, A)$ into a directed acyclic graph~$G'$ and by adapting the dynamic program from the proof of \cref{thm:adp_poly_acyclic_graphs} to the new situation.
	To this end, let $n = |V|$.

	The vertex set of~$G'$ consists of $n$ copies $v_1, \dots, v_n$ for every vertex $v \in V$, and we call the vertices $\{v_i \setcolon v \in V\}$ the \emph{$i$-th layer} of~$G'$.
	For every arc~$uv \in A$ we add the $n - 1$ arcs~$u_{i - 1} v_i$ for $i \in \{2, \dots, n\}$ to~$G'$, which we call \emph{copies} of~$uv$.
	Since all arcs in~$G'$ lead exactly one layer up, the graph~$G'$ constructed so far is acyclic.
	Furthermore, we add an additional vertex~$t'$ to~$G'$, which we connect with arcs $t_i t'$ for $i \in \{1, \dots, n\}$.
	Note that~$G'$ remains acyclic.
	In order to ensure \cref{ass:every-arc-in-st-path} we restrict~$G'$ to those vertices and arcs that are reachable from~$s_1$ and from which we can reach~$t'$.

	The basic idea is now to find $k$ almost disjoint $s_1$-$t'$-paths in~$G'$ and translate these back to $s$-$t$-paths in the original graph~$G$.
	For this we ignore the last vertex~$t'$ and replace every other vertex~$v_i$ on such a path by the corresponding vertex~$v \in V$.
	In this way, however, almost disjoint paths in~$G'$ need not remain almost disjoint in~$G$.
	For this to be the case, we have to identify all copies of an arc:
	for every two paths that we choose in~$G'$ there must be at most one arc~$uv \in A$ of which both paths contain a copy.
	If this is the case, we call the paths \emph{almost copy-disjoint}.

	To achieve this, we have to slightly modify the dynamic program from the proof of \cref{thm:adp_poly_acyclic_graphs}.
	More precisely, we update the definition of the Boolean value~$x((a_1, \dots, a_k), \cI)$.
	Instead of assuming the $s$-$\omega(a_i)$-paths~$P_i$ to be almost disjoint, we now require that they are almost copy-disjoint.
	Accordingly, we have to update the interpretation of the intersection pattern:
	we now have $\{i, j\} \in \cI$ if and only if the paths $P_i$ and~$P_j$ both contain a copy of the same arc.

	Consequently, we have to amend the recursion.
	A state must be \false not only if $a_1 = a_i$ and $\{1, i\} \notin \cI$, but even if $a_1$ and~$a_i$ are copies of the same arc and $\{1, i\} \notin \cI$.
	This also entails a slightly different definition of~$\cC$:
	\begin{equation*}
	\cC = \{\{1, i\} \setcolon \text{$a_1$ and $a_i$ are copies of the same arc}, i \neq 1\}.
	\end{equation*}

	Overall, the size of~$G'$ is polynomial in the size of~$G$, it can be constructed in polynomial time, and all modifications in the dynamic program induce only polynomial overhead.
	\qed
\end{proof}

\subsection{\NP-completeness} \label{sec:adp:completeness}

Although we can solve \ADP for constant~$k$ in polynomial time, it is \NP-complete in general.
This is stated in \cref{thm:adpnp}, which we prove in this section.
\adpnp*
\ADP is contained in \NP since we can check in polynomial time whether $k$~paths are almost disjoint.
To prove the hardness we reduce the \NP-complete independent set problem \cite{GJ79} to \ADP.
For this let an instance of the independent set problem be given by an undirected graph $H = (V_H, E_H)$.
After constructing a directed acyclic graph~$G = (V, A)$ we show that $H$ has an independent set of size~$k$ if and only if there are $2 \cdot |E_H| + k$ almost disjoint $s$-$t$-paths in~$G$.

\subsection*{Graph Construction}

The basic component to construct an instance $G = (V, A)$ of \ADP from~$H$ is a gadget $\gadget{e}$ for every edge $e \in E_H$.
We first describe these gadgets and how we combine them to obtain the graph~$G$.

\subsubsection{The Gadget.}
The main component of the \ADP instance is the edge gadget depicted in \cref{fig:adp:gadget}.
Such a gadget $\gadget{uv}$ corresponds to an edge $u v \in E_H$ and has four input vertices (drawn as arrows pointing into the gadget):
two labeled $u$ and~$v$ corresponding to the end vertices of the edge and two auxiliary inputs $h_1$ and~$h_2$.
Analogously, the gadget also has four output vertices (shown as arrows pointing outward):
$u'$ and~$v'$ as well as $h'_1$ and~$h'_2$.
In addition, it contains ten interior vertices, which we name and connect as drawn in \cref{fig:adp:gadget}.
When using such gadgets to construct $G$, we use the schematic illustration depicted on the right of \cref{fig:adp:rearrange_gadget_inputs}; note that the in- and output vertices are rearranged.

\begin{figure}[p]
	\centering
	\begin{tikzpicture}[scale=.8]
		\draw[contourLine] (0,2) -- (11,2) -- (11,-2) -- (0,-2) -- cycle;
	
		\coordinate (cu) at (0,1);
		\coordinate (cv) at (0,-1);
		\coordinate (ch1) at (1,2);
		\coordinate (ch2) at (1,-2);
		\coordinate (cw1) at (1,1);
		\coordinate (cw2) at (1,-1);
		\coordinate (cx1) at (3,1);
		\coordinate (cx2) at (3,-1);
		\coordinate (cy1) at (4.5,0);
		\coordinate (cy2) at (6.5,0);
		\coordinate (cz1) at (8,1);
		\coordinate (cz2) at (8,-1);
		\coordinate (ca1) at (10,1);
		\coordinate (ca2) at (10,-1);
		\coordinate (cup) at (11,1);
		\coordinate (cvp) at (11,-1);
		\coordinate (chp2) at (10,2);
		\coordinate (chp1) at (10,-2);
	
		\node[arrowVertex] (u) at (cu) [label=left: $u$] {};
		\node[arrowVertex] (v) at (cv) [label=left: $v$] {};
		\node[arrowVertex, rotate=-90] (h1) at (ch1) [label=left: $h_1$] {};
		\node[arrowVertex, rotate=90] (h2) at (ch2) [label=left: $h_2$] {};
		\node[stdVertex] (w1) at (cw1) [label=below: $x^L_1$] {};
		\node[stdVertex] (w2) at (cw2) [label=above: $x^L_2$] {};
		\node[stdVertex] (x1) at (cx1) [label=below left: $y^L_1$] {};
		\node[stdVertex] (x2) at (cx2) [label=above left: $y^L_2$] {};
		\node[stdVertex] (y1) at (cy1) [label=above right: $z^L$] {};
		\node[stdVertex] (y2) at (cy2) [label=above left: $z^R$] {};
		\node[stdVertex] (z1) at (cz1) [label=below right: $y^R_1$] {};
		\node[stdVertex] (z2) at (cz2) [label=above right: $y^R_2$] {};
		\node[stdVertex] (a1) at (ca1) [label=below: $x^R_1$] {};
		\node[stdVertex] (a2) at (ca2) [label=above: $x^R_2$] {};
		\node[arrowVertex] (up) at (cup) [label=right: $u'$] {};
		\node[arrowVertex] (vp) at (cvp) [label=right: $v'$] {};
		\node[arrowVertex, rotate=-90] (hp1) at (chp1) [label=right: $h'_1$] {};
		\node[arrowVertex, rotate=90] (hp2) at (chp2) [label=right: $h'_2$] {};
	
		\draw[->]
			(u.center) edge (w1)
			(h1.center) edge (w1)
			(v.center) edge (w2)
			(h2.center) to (w2)
			(w1) edge (x1)
			(w2) edge (x2)
			(x1) edge (y1) edge [bend left=40pt] (z2)
			(x2) edge (y1) edge [bend right=40pt] (z1)
			(y1) edge (y2)
			(y2) edge (z1) edge (z2)
			(z1) edge (a1)
			(z2) edge (a2)
			(a1) edge[-] (hp2.center) edge[-] (up.center)
			(a2) edge[-] (hp1.center) edge[-] (vp.center);
	\end{tikzpicture}	\caption[Gadget for the hardness proof of \ADP.]{%
		The gadget $\gadget{uv}$ of an edge~$u v \in E_H$.
	}
	\label{fig:adp:gadget}
\end{figure}
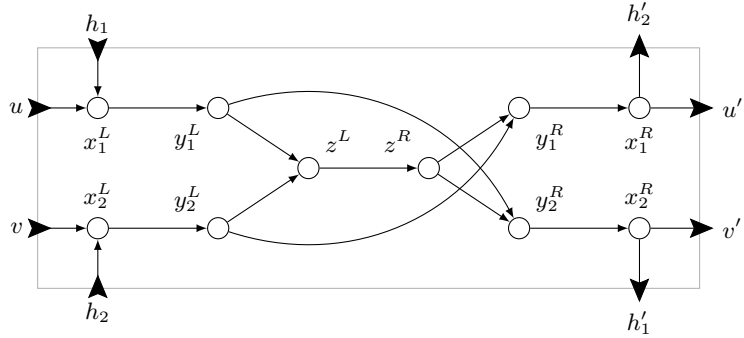

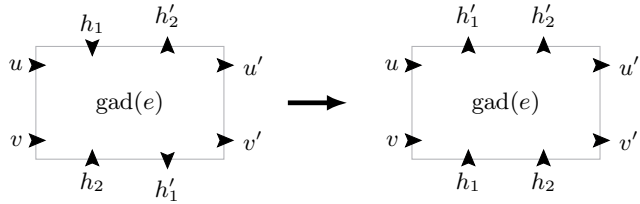
\begin{figure}[p]
	\centering
	\begin{tikzpicture}
		\draw[contourLine]
			(-1.25,.75) --
	        (1.25,.75) --
	        (1.25,-.75) --
	        (-1.25,-.75) -- cycle;
		\node at (0,0) {$\gadget{e}$};
		\node[smallArrowVertex] at (-1.25,.5) [label=left: $u$] {};
		\node[smallArrowVertex] at (-1.25,-.5) [label=left: $v$] {};
		\node[smallArrowVertex, rotate=-90] at (-.5,.75) [label=left: $h_1$] {};
		\node[smallArrowVertex, rotate=90] at (-.5,-.75) [label=left: $h_2$] {};
		\node[smallArrowVertex] at (1.25,.5) [label=right: $u'$] {};
		\node[smallArrowVertex] at (1.25,-.5) [label=right: $v'$] {};
		\node[smallArrowVertex, rotate=-90] at (.5,-.75) [label=right: $h'_1$] {};
		\node[smallArrowVertex, rotate=90] at (.5,.75) [label=right: $h'_2$] {};
	
		\draw[-{Latex[length=8pt,width=6pt]}, line width=1.8pt] (2.1,0) to (2.9,0);
	
		\draw[contourLine]
			(3.75,.75) --
	        (6.25,.75) --
	        (6.25,-.75) --
	        (3.75,-.75) -- cycle;
		\node at (5,0) {$\gadget{e}$};
		\node[smallArrowVertex] at (3.75,.5) [label=left: $u$] {};
		\node[smallArrowVertex] at (3.75,-.5) [label=left: $v$] {};
		\node[smallArrowVertex, rotate=90] at (4.5,-.75) [label=left: $h_1$] {};
		\node[smallArrowVertex, rotate=90] at (5.5,-.75) [label=left: $h_2$] {};
		\node[smallArrowVertex] at (6.25,.5) [label=right: $u'$] {};
		\node[smallArrowVertex] at (6.25,-.5) [label=right: $v'$] {};
		\node[smallArrowVertex, rotate=90] at (4.5,.75) [label=right: $h'_1$] {};
		\node[smallArrowVertex, rotate=90] at (5.5,.75) [label=right: $h'_2$] {};
	\end{tikzpicture}	\caption[Rearrangement of gadget in- and outputs.]{%
		The rearrangement of the in- and outputs of a gadget from \cref{fig:adp:gadget} in order to simplify the drawing of the graph in \cref{fig:adp:hardness_instance}.
	}
	\label{fig:adp:rearrange_gadget_inputs}
\end{figure}

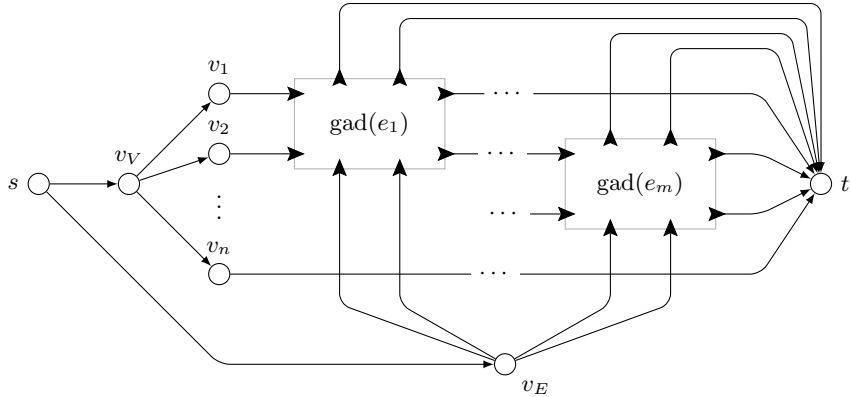
\begin{figure}[p]
	\centering
	\begin{tikzpicture}[scale=.8]
		\newcommand{\drawGadget}[5]{
			\draw[contourLine]
				({#2-1.25},{#3+.75}) --
				({#2+1.25},{#3+.75}) --
				({#2+1.25},{#3-.75}) --
				({#2-1.25},{#3-.75}) -- cycle;
			\node (G#1) at (#2,#3) {$\gadget{e_{#1}}$};
			\node[smallArrowVertex] (uG#1) at ({#2-1.25},{#3+.5}) {};
			\node[smallArrowVertex] (vG#1) at ({#2-1.25},{#3-.5}) {};
			\node[smallArrowVertex, rotate=90] (h1G#1) at ({#2-.5},{#3-.75}) {};
			\node[smallArrowVertex, rotate=90] (h2G#1) at ({#2+.5},{#3-.75}) {};
			\node[smallArrowVertex] (upG#1) at ({#2+1.25},{#3+.5}) {};
			\node[smallArrowVertex] (vpG#1) at ({#2+1.25},{#3-.5}) {};
			\node[smallArrowVertex, rotate=90] (h1pG#1) at ({#2-.5},{#3+.75}) {};
			\node[smallArrowVertex, rotate=90] (h2pG#1) at ({#2+.5},{#3+.75}) {};
			\coordinate (cbh1G#1) at ({#2-.5},#4);
			\coordinate (cth1G#1) at ({#2-.5},#5);
			\coordinate (cbh2G#1) at ({#2+.5},#4);
			\coordinate (cth2G#1) at ({#2+.5},#5);
		}
	
		\def\ct{2.5}
		\def\cb{-2}
		\node[stdVertex] (s) at (0,0) [label=left: $s$] {};
		\node[stdVertex] (vV) at (1.5,0) [label=above: $v_V$] {};
		\node[stdVertex] (v1) at (3,1.5) [label=above: $v_1$] {};
		\node[stdVertex] (v2) at (3,.5) [label=above: $v_2$] {};
		\node (vk) at (3,-.25) {$\vdots$};
		\node[stdVertex] (vn) at (3,-1.5) [label=above: $v_n$] {};
	
		\drawGadget{1}{5.5}{1}{\cb}{\ct}
		\drawGadget{m}{10}{0}{\cb}{\ct}
	
		\node[stdVertex] (vE) at (7.75,-3) [label=below right: $v_E$] {};
		\node[stdVertex] (t) at (13,0) [label=right: $t$] {};
	
		\coordinate (cv1) at (12,1.5);
		\coordinate (cv2) at (12,.5);
		\coordinate (cv3) at (12,-.5);
		\coordinate (cvn) at (12,-1.5);
	
		\draw[->] (s) edge (vV)
	        (vV) edge (v1) edge (v2) edge (vn)
	        (v1) edge [-] (uG1)
	        (v2) edge [-] (vG1);
		\draw[->] (upG1.center) [rounded corners] -- node [xshift=-1.25cm, fill=white] {$\cdots$} (cv1) -- (t);
		\draw (vpG1.center) -- node [fill=white] {$\cdots$} (uGm);
		\draw[->] (upGm.center) [rounded corners] -- (cv2) -- (t);
	    \draw[->] (vpGm.center) [rounded corners] -- (cv3) -- (t);
	    \draw[->] (vn) [rounded corners] -- node [fill=white] {$\cdots$} (cvn) -- (t);
		\node[fill=white] (n) at ($(vGm) + (-1,0)$) {$\cdots$};
		\draw (n) -- (vGm);
	
		\draw[->, rounded corners] (s) -- (3, -3) -- (vE);
		\draw
			(vE) [rounded corners] -- (cbh1G1) -- (h1G1)
			(vE) [rounded corners] -- (cbh2G1) -- (h2G1)
			(vE) [rounded corners] -- (cbh1Gm) -- (h1Gm)
			(vE) [rounded corners] -- (cbh2Gm) -- (h2Gm);
		\draw[->] (h1pG1.center) [rounded corners] -- ++(0, 1.25) -- (13, 3) -- (t);
		\draw[->] (h2pG1.center) [rounded corners] -- ++(0, 1) -- (12.75, 2.75) -- (t);
		\draw[->] (h1pGm.center) [rounded corners] -- ++(0, 1.75) -- (12.5, 2.5) -- (t);
		\draw[->] (h2pGm.center) [rounded corners] -- ++(0, 1.5) -- (12.25, 2.25) -- (t);
	\end{tikzpicture}	\caption[Graph for the hardness proof of \ADP.]{%
		The graph $G = (V, A)$ for the hardness proof of \ADP. The gadgets are those from \cref{fig:adp:gadget} with rearranged in- and outputs as specified in \cref{fig:adp:rearrange_gadget_inputs}.
	}
	\label{fig:adp:hardness_instance}
\end{figure}

\subsubsection{The Graph.}
The graph $G = (V, A)$ of the \ADP instance corresponding to the graph $H = (V_H, E_H)$ is drawn in \cref{fig:adp:hardness_instance}.
It consists of a gadget $\gadget{e}$ for every edge $e \in E_H$ and additional vertices $V_H \cup \{s, t, v_V, v_E\}$.
The source~$s$ is connected with $v_V$ and~$v_E$ and the vertex~$v_V$ has outgoing arcs to all $v \in V_H$.
To every auxiliary input of a gadget we have an arc from~$v_E$.
From every auxiliary output of a gadget there is an arc to the target~$t$.

Finally, we explain how the vertex inputs of the gadgets are connected.
To this end, sort the edges $E_H = \{e_1, \dots, e_m\}$ arbitrarily.
In the graph~$G$, every vertex $u \in V_H$ is connected to the target~$t$ by a path~$P_u$ that starts with $(s, v_V, u)$ and passes through every gadget~$\gadget{e}$ corresponding to an incident edge~$e \in \incidentedges[H]{u}$.
For $\ell = \degree[H]{u}$ we choose $j_1 < \dots < j_{\ell}$ such that $\incidentedges[H]{u} = \{e_{j_1}, \dots, e_{j_{\ell}}\}$.
We connect $u$ with the input of $\gadget{e_{j_1}}$ that is labeled~$u$.
Its output~$u'$ is connected with the input~$u$ of $\gadget{e_{j_2}}$ and so on.
Finally, the output~$u'$ of the last gadget $\gadget{e_{j_{\ell}}}$ has an arc to the target.
If the vertex~$u \in V_H$ has no incident edge, we introduce the direct arc $u t$.

\begin{definition}
	\label{def:adp:auxiliary_path}
	\label{def:adp:vertex_path}
	Every $s$-$t$-path in~$G$ either starts with the arc~$s v_V$ or with the arc~$s v_E$.
	Those starting with $s v_V$ are \emph{vertex paths} and those starting with $s v_E$ are \emph{auxiliary paths}.
\end{definition}

\subsection*{From an Independent Set to Almost Disjoint Paths}

\begin{lemma} \label{lem:adp:size_ind_set_many_almost_disjoint_paths}
	Given an independent set $U \subseteq V_H$ in $H$ of size $|U| = k$, there are $2 m + k$ almost disjoint $s$-$t$-paths in $G$.
\end{lemma}
\begin{proof}
	We construct $2 m + k$ almost disjoint $s$-$t$-paths, $k$ of which are vertex paths for the vertices in~$U$ and the remaining $2 m$ are auxiliary paths.

	The auxiliary paths are obtained by extending the $h_1$-$h'_1$- and $h_2$-$h'_2$-paths visualized in \cref{fig:adp:gadget_paths} of all gadgets $\gadget{e}$, $e \in E_H$.
	They have the first arc~$s v_E$ in common and are disjoint afterwards.

	For $u \in U$ we choose the path~$P_u$ from the graph construction.
	It is the unique $s$-$t$-path that starts with $(s, v_V, u)$ and uses all $u$-$u'$-paths through gadgets~$\gadget{e}$ of incident edges $e \in \incidentedges[H]{u}$ as well as the arcs connecting these.
	Since $U \subseteq V_H$ is an independent set in~$H$ and since the gadgets correspond to edges in~$H$, there is no gadget $\gadget{u v}$ with $\{u, v\} \subseteq U$.
	Thus, for every gadget, we choose at most one vertex path passing through it.
	This implies that also all chosen vertex paths have the first arc $s v_V$ in common and are disjoint afterwards.

	A chosen vertex path~$P_u$ and a chosen auxiliary path have an arc in common if and only if the auxiliary path passes through a gadget corresponding to an edge incident to~$u$.
	Hence, the $2 m + k$ chosen $s$-$t$-paths are almost disjoint.
	\qed
\end{proof}

\vspace{-7pt}

\begin{figure}[h]
	\centering
	\begin{tikzpicture}[scale=.8]
		\draw[contourLine] (0,2) -- (11,2) -- (11,-2) -- (0,-2) -- cycle;
	
		\coordinate (cu) at (0,1);
		\coordinate (cv) at (0,-1);
		\coordinate (ch1) at (1,2);
		\coordinate (ch2) at (1,-2);
		\coordinate (cw1) at (1,1);
		\coordinate (cw2) at (1,-1);
		\coordinate (cx1) at (3,1);
		\coordinate (cx2) at (3,-1);
		\coordinate (cy1) at (4.5,0);
		\coordinate (cy2) at (6.5,0);
		\coordinate (cz1) at (8,1);
		\coordinate (cz2) at (8,-1);
		\coordinate (ca1) at (10,1);
		\coordinate (ca2) at (10,-1);
		\coordinate (cup) at (11,1);
		\coordinate (cvp) at (11,-1);
		\coordinate (chp2) at (10,2);
		\coordinate (chp1) at (10,-2);
	
		\draw[pathMarker1] (cu) --
		    (cw1) -- (cx1) -- ($(cy1) + (0,4pt)$) -- ($(cy2) + (0,4pt)$) --
		    (cz1) -- (ca1) -- (cup);
		\draw[pathMarker2] (cv) --
		    (cw2) -- (cx2) -- ($(cy1) + (0,-4pt)$) -- ($(cy2) + (0,-4pt)$) --
		    (cz2) -- (ca2) -- (cvp);
	
		\node[arrowVertex] (u) at (cu) [label=left: $u$] {};
		\node[arrowVertex] (v) at (cv) [label=left: $v$] {};
		\node[arrowVertex, rotate=-90] (h1) at (ch1) [label=left: $h_1$] {};
		\node[arrowVertex, rotate=90] (h2) at (ch2) [label=left: $h_2$] {};
		\node[stdVertex] (w1) at (cw1) {};
		\node[stdVertex] (w2) at (cw2) {};
		\node[stdVertex] (x1) at (cx1) {};
		\node[stdVertex] (x2) at (cx2) {};
		\node[stdVertex] (y1) at (cy1) {};
		\node[stdVertex] (y2) at (cy2) {};
		\node[stdVertex] (z1) at (cz1) {};
		\node[stdVertex] (z2) at (cz2) {};
		\node[stdVertex] (a1) at (ca1) {};
		\node[stdVertex] (a2) at (ca2) {};
		\node[arrowVertex] (up) at (cup) [label=right: $u'$] {};
		\node[arrowVertex] (vp) at (cvp) [label=right: $v'$] {};
		\node[arrowVertex, rotate=-90] (hp1) at (chp1) [label=right: $h'_1$] {};
		\node[arrowVertex, rotate=90] (hp2) at (chp2) [label=right: $h'_2$] {};
	
		\draw[->]
			(u.center) edge (w1)
			(v.center) to (w2)
			(w1) edge [ultra thick, color=pathColor3] (x1)
			(w2) edge [ultra thick, color=pathColor4] (x2)
			(x1) edge (y1) edge [bend left=40pt, ultra thick, color=pathColor3] (z2)
			(x2) edge (y1) edge [bend right=40pt, ultra thick, color=pathColor4] (z1)
			(y1) edge (y2)
			(y2) edge (z1) edge (z2)
			(z1) edge [ultra thick, color=pathColor4] (a1)
			(z2) edge [ultra thick, color=pathColor3] (a2)
			(a1) edge [-, ultra thick, color=pathColor4] (hp2) edge[-] (up)
			(a2) edge [-, ultra thick, color=pathColor3] (hp1) edge[-] (vp);
		\draw[->,ultra thick, color=pathColor4] (h2.center) to (w2);
		\draw[->,ultra thick, color=pathColor3] (h1.center) to (w1);

		\node[arrowVertex, rotate=-90] (h1) at (ch1) [label=left: $h_1$] {};
		\node[arrowVertex, rotate=90] (h2) at (ch2) [label=left: $h_2$] {};
		\node[arrowVertex, rotate=-90] (hp1) at (chp1) [label=right: $h'_1$] {};
		\node[arrowVertex, rotate=90] (hp2) at (chp2) [label=right: $h'_2$] {};
	\end{tikzpicture}	\vspace{-7pt}
	\caption[Paths through a gadget.]{%
		A gadget as in \cref{fig:adp:gadget} with four paths through it: the unique $u$-$u'$-path (red), the unique $v$-$v'$-path (blue), an $h_1$-$h'_1$-path (green), and an $h_2$-$h'_2$-path (orange).
	}
	\label{fig:adp:gadget_paths}
\end{figure}
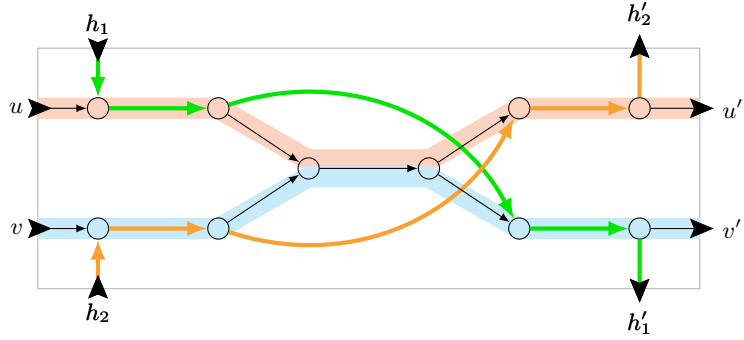

\vspace{-7pt}

\subsection*{From Almost Disjoint Paths to an Independent Set}

In the following, let $\cQ$ be a set of $2 m + k$ almost disjoint $s$-$t$-paths in~$G$ among which the number of auxiliary paths is maximized.
We assume $k \geq 0$ as we can always choose $2 m$ auxiliary paths as described in the proof of \cref{lem:adp:size_ind_set_many_almost_disjoint_paths}.

\begin{assumption} \label{ass:adp_with_max_aux_paths}
	No set of $2 m + k$ almost disjoint $s$-$t$-paths in~$G$ contains more auxiliary paths than $\cQ$.
\end{assumption}

The following three lemmas provide structural results of the paths in~$\cQ$.
They allow us to prove the counterpart of \cref{lem:adp:size_ind_set_many_almost_disjoint_paths} in \cref{lem:adp:number_almost_disjoint_paths_large_ind_set}, which completes the proof of \cref{thm:adp_np-complete}.

\pagebreak[2]

\begin{lemma} \label{lem:adp:aux_path_aux_output}
	Without loss of generality we can assume that every auxiliary path in~$\cQ$ leaves the first gadget it enters via an auxiliary output, which leads it directly to~$t$.
\end{lemma}
\begin{proof}
	Let $P \in \cQ$ be an auxiliary path, let $\gadget{e}$ be the first gadget it enters, and suppose that~$P$ leaves $\gadget{e}$ via a vertex output~$u'$.
	We assume that $\cQ$ is chosen such that it minimizes the number of such paths (among those sets $\cQ$ satisfying \cref{ass:adp_with_max_aux_paths}).
	By the construction of the graph, there is a single arc leaving~$u'$.
	This arc either points to the vertex input of another gadget or to the target.

	We first consider the case that the arc leaving~$u'$ points to a vertex input~$\tilde{u}$ of another gadget~$\gadget{\tilde{e}}$.
	In this situation, the path~$P$ enters $\gadget{\tilde{e}}$ via~$\tilde{u}$ directly after leaving~$\gadget{e}$ via~$u'$.
	It is depicted in \cref{fig:adp:aux_path_aux_output}.
	In this case, no auxiliary path leaves $\gadget{e}$ via~$h'_2$ and no auxiliary path enters $\gadget{\tilde{e}}$ via~$\tilde{h}_1$.
	Otherwise, such a path has not only the arc $s v_E$ in common with~$P$, but also either $y^R_1 x^R_1$ or $\tilde{x}^L_1 \tilde{y}^L_1$.
	Thus, we can replace~$P$ in~$\cQ$ by two auxiliary paths:
	one that equals~$P$ until vertex~$x^R_1$ but then continues along $(x^R_1, h'_2, t)$ and the other starting with $(s, v_E, \tilde{h}'_1, \tilde{x}^L_1)$ and following~$P$ from~$\tilde{x}^L_1$ on.
	To preserve the number of paths in~$\cQ$, we remove a vertex path from~$\cQ$ in return.

	Note that the paths from~$\cQ$ remain almost disjoint after this modification, except if there is a vertex path leaving~$\gadget{e}$ via~$h'_2$.
	However, if this is the case we can simple remove this vertex path.
	Also note that~$\cQ$ contains at least one vertex path since $k \geq 0$ and because~$\cQ$ contains at most $2 m - 1$ auxiliary paths:
	$\outdegree{v_E} = 2 m$ and no auxiliary path uses the arc $v_E \tilde{h}_1$.
	Thus, the replacement of~$P$ in~$\cQ$ contradicts \cref{ass:adp_with_max_aux_paths} such that this case cannot occur.

	We now consider the remaining case, in which the arc leaving~$u'$ directly points to the target~$t$.
	In this situation, $P$ ends with $(y^R_1, x^R_1, u', t)$ and we can modify~$P$ by using $(y^R_1, x^R_1, h'_2, t)$ instead.
	As argued in the first case, another path in~$\cQ$ leaving $\gadget{e}$ via~$h'_2$ has to be a vertex path.
	If it exists, we can modify it by changing its end from $(y^R_1, x^R_1, h'_2, t)$ to $(y^R_1, x^R_1, u', t)$.

	Hence, these modifications reduce the number of auxiliary paths leaving a gadget via a vertex output by one, which contradicts the assumption on $\cQ$ we made at the start of the proof.
	\qed
\end{proof}

\begin{figure}[h]
	\centering
	\begin{tikzpicture}[scale=.8]
		\begin{scope}
			\clip (6.8, 2.75) rectangle (12, -2.75);
			\draw[contourLine] (0,2) -- (11,2) -- (11,-2) -- (0,-2) -- cycle;
	
			\coordinate (cw1) at (1,1);
			\coordinate (cw2) at (1,-1);
			\coordinate (cx1) at (3,1);
			\coordinate (cx2) at (3,-1);
			\coordinate (cy1) at (4.5,0);
			\coordinate (cy2) at (6.5,0);
			\coordinate (cz1) at (8,1);
			\coordinate (cz2) at (8,-1);
			\coordinate (ca1) at (10,1);
			\coordinate (ca2) at (10,-1);
			\coordinate (cup) at (11,1);
			\coordinate (cvp) at (11,-1);
			\coordinate (chp2) at (10,2);
			\coordinate (chp1) at (10,-2);
	
			\node[stdVertex] (w1) at (cw1) [label=below: $x^L_1$] {};
			\node[stdVertex] (w2) at (cw2) [label=above: $x^L_2$] {};
			\node[stdVertex] (x1) at (cx1) [label=below left: $y^L_1$] {};
			\node[stdVertex] (x2) at (cx2) [label=above left: $y^L_2$] {};
			\node[stdVertex] (y1) at (cy1) [label=above right: $z^L$] {};
			\node[stdVertex] (y2) at (cy2) [label=above left: $z^R$] {};
			\node[stdVertex] (z1) at (cz1) [label=below right: $y^R_1$] {};
			\node[stdVertex] (z2) at (cz2) [label=above right: $y^R_2$] {};
			\node[stdVertex] (a1) at (ca1) [label=below: $x^R_1$] {};
			\node[stdVertex] (a2) at (ca2) [label=above: $x^R_2$] {};
			\node[arrowVertex] (up) at (cup) [label=above right: $u'$] {};
			\node[arrowVertex] (vp) at (cvp) [label=below right: $v'$] {};
			\node[arrowVertex, rotate=-90] (hp1) at (chp1) [label=right: $h'_1$] {};
			\node[arrowVertex, rotate=90] (hp2) at (chp2) [label=right: $h'_2$] {};
	
			\draw[->]
				(w1) edge (x1)
				(w2) edge (x2)
				(x1) edge (y1) edge [bend left=40pt] (z2)
				(x2) edge (y1) edge [bend right=40pt] (z1)
				(y1) edge (y2)
				(y2) edge (z1) edge (z2)
				(z1) edge (a1)
				(z2) edge (a2)
				(a1) edge[-] (hp2) edge[-] (up)
				(a2) edge[-] (hp1) edge[-] (vp);
		\end{scope}
	
		\begin{scope}
			\clip (13, 2.75) rectangle (18.2, -2.75);
			\draw[contourLine] (14,2) -- (25,2) -- (25,-2) -- (14,-2) -- cycle;
	
			\coordinate (cur) at (14,1);
			\coordinate (cvr) at (14,-1);
			\coordinate (ch1r) at (15,2);
			\coordinate (ch2r) at (15,-2);
			\coordinate (cw1r) at (15,1);
			\coordinate (cw2r) at (15,-1);
			\coordinate (cx1r) at (17,1);
			\coordinate (cx2r) at (17,-1);
			\coordinate (cy1r) at (18.5,0);
			\coordinate (cy2r) at (20.5,0);
			\coordinate (cz1r) at (22,1);
			\coordinate (cz2r) at (22,-1);
			\coordinate (ca1r) at (24,1);
			\coordinate (ca2r) at (24,-1);
	
			\node[arrowVertex] (u) at (cur) [label=above left: $\tilde{u}$] {};
			\node[arrowVertex] (v) at (cvr) [label=below left: $\tilde{v}$] {};
			\node[arrowVertex, rotate=-90] (h1) at (ch1r) [label=left: $\tilde{h}_1$] {};
			\node[arrowVertex, rotate=90] (h2) at (ch2r) [label=left: $\tilde{h}_2$] {};
			\node[stdVertex] (w1) at (cw1r) [label=below: $\tilde{x}^L_1$] {};
			\node[stdVertex] (w2) at (cw2r) [label=above: $\tilde{x}^L_2$] {};
			\node[stdVertex] (x1) at (cx1r) [label=below left: $\tilde{y}^L_1$] {};
			\node[stdVertex] (x2) at (cx2r) [label=above left: $\tilde{y}^L_2$] {};
			\node[stdVertex] (y1) at (cy1r) [label=above right: $\tilde{z}^L$] {};
			\node[stdVertex] (y2) at (cy2r) [label=above left: $\tilde{z}^R$] {};
			\node[stdVertex] (z1) at (cz1r) [label=below right: $\tilde{y}^R_1$] {};
			\node[stdVertex] (z2) at (cz2r) [label=above right: $\tilde{y}^R_2$] {};
			\node[stdVertex] (a1) at (ca1r) [label=below: $\tilde{x}^R_1$] {};
			\node[stdVertex] (a2) at (ca2r) [label=above: $\tilde{x}^R_2$] {};
	
			\draw[->] (u.center) edge (w1)
				(h1.center) edge (w1)
				(v.center) edge (w2)
				(h2.center) to (w2)
				(w1) edge (x1)
				(w2) edge (x2)
				(x1) edge (y1) edge [bend left=40pt] (z2)
				(x2) edge (y1) edge [bend right=40pt] (z1)
				(y1) edge (y2)
				(y2) edge (z1) edge (z2)
				(z1) edge (a1)
				(z2) edge (a2);
		\end{scope}
	
		\draw[-] (up.center) edge (u);
	
		\begin{scope}[on background layer]
			\draw[pathMarker1, line cap=round] (cz1) -- (ca1) -- (cup) -- (cur) -- (cw1r) -- (cx1r);
		\end{scope}
	\end{tikzpicture}	\caption[A path using a vertex output.]{A path that leaves $\gadget{e}$ via its vertex output~$u'$ and enters $\gadget{\tilde{e}}$ via its vertex input~$\tilde{u}$ uses at least the red marked arcs.}
	\label{fig:adp:aux_path_aux_output}
\end{figure}
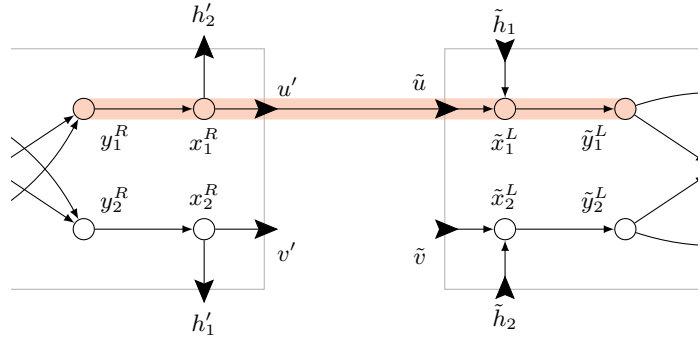

\begin{lemma} \label{lem:adp:vertex_path_correct_vertex_output}
	Let $\gadget{e}$ be a gadget that is passed through by exactly one vertex path~$P$ from~$\cQ$.
	If $P$ enters $\gadget{e}$ via a vertex input~$u$, it leaves $\gadget{e}$ via the corresponding vertex output~$u'$.
\end{lemma}
\begin{proof}
	If $P$ leaves $\gadget{e}$ via an auxiliary output, \cref{lem:adp:aux_path_aux_output} implies that $\cQ$ contains at most one auxiliary path passing through~$\gadget{e}$.
	In this case, we can remove all paths passing through $\gadget{e}$ from~$\cQ$ and replace them by the same number of auxiliary paths passing through~$\gadget{e}$.
	This increases the number of auxiliary paths in~$\cQ$ contradicting \cref{ass:adp_with_max_aux_paths}.

	Next, suppose that $P$ leaves $\gadget{e}$ via the vertex output~$v'$.
	In this case, it must definitely use the arcs that are marked in red in \cref{fig:adp:vertex_path_correct_vertex_output}.
	The only chance for an almost disjoint auxiliary path~$P'$ entering $\gadget{e}$ via~$h_1$ is to use the $h_1$-$h'_2$-path whose arcs are green in \cref{fig:adp:vertex_path_correct_vertex_output}.
	However, every other auxiliary path entering $\gadget{e}$ must share an arc of $\gadget{e}$ with~$P'$.
	Thus, they are not almost disjoint and $\cQ$ contains again at most one auxiliary path passing through~$\gadget{e}$.
	As in the first case we can replace all paths through~$\gadget{e}$, thereby increasing the number of auxiliary paths in~$\cQ$, and again contradicting \cref{ass:adp_with_max_aux_paths}.
	\qed
\end{proof}

\begin{figure}
	\centering
	\begin{tikzpicture}[scale=.8]
		\clip (-1,-2.75) rectangle (12,2.75);
		\draw[contourLine] (0,2) -- (11,2) -- (11,-2) -- (0,-2) -- cycle;
	
		\coordinate (cu) at (0,1);
		\coordinate (cv) at (0,-1);
		\coordinate (ch1) at (1,2);
		\coordinate (ch2) at (1,-2);
		\coordinate (cw1) at (1,1);
		\coordinate (cw2) at (1,-1);
		\coordinate (cx1) at (3,1);
		\coordinate (cx2) at (3,-1);
		\coordinate (cy1) at (4.5,0);
		\coordinate (cy2) at (6.5,0);
		\coordinate (cz1) at (8,1);
		\coordinate (cz2) at (8,-1);
		\coordinate (ca1) at (10,1);
		\coordinate (ca2) at (10,-1);
		\coordinate (cup) at (11,1);
		\coordinate (cvp) at (11,-1);
		\coordinate (chp2) at (10,2);
		\coordinate (chp1) at (10,-2);
	
		\draw[pathMarker1]
			(cu) -- (cw1)
			(ca2) -- (cvp);
		\draw[pathMarker1, line cap=round]
			(cw1) -- (cx1)
			(cz2) -- (ca2);
	
		\node[arrowVertex] (u) at (cu) [label=left: $u$] {};
		\node[arrowVertex] (v) at (cv) [label=left: $v$] {};
		\node[arrowVertex, rotate=-90] (h1) at (ch1) [label=left: $h_1$] {};
		\node[arrowVertex, rotate=90] (h2) at (ch2) [label=left: $h_2$] {};
		\node[stdVertex] (w1) at (cw1) [label=below: $x^L_1$] {};
		\node[stdVertex] (w2) at (cw2) [label=above: $x^L_2$] {};
		\node[stdVertex] (x1) at (cx1) [label=below left: $y^L_1$] {};
		\node[stdVertex] (x2) at (cx2) [label=above left: $y^L_2$] {};
		\node[stdVertex] (y1) at (cy1) [label=above right: $z^L$] {};
		\node[stdVertex] (y2) at (cy2) [label=above left: $z^R$] {};
		\node[stdVertex] (z1) at (cz1) [label=below right: $y^R_1$] {};
		\node[stdVertex] (z2) at (cz2) [label=above right: $y^R_2$] {};
		\node[stdVertex] (a1) at (ca1) [label=below: $x^R_1$] {};
		\node[stdVertex] (a2) at (ca2) [label=above: $x^R_2$] {};
		\node[arrowVertex] (up) at (cup) [label=right: $u'$] {};
		\node[arrowVertex] (vp) at (cvp) [label=right: $v'$] {};
		\node[arrowVertex, rotate=-90] (hp1) at (chp1) [label=right: $h'_1$] {};
		\node[arrowVertex, rotate=90] (hp2) at (chp2) [label=right: $h'_2$] {};
	
		\draw[->] (u.center) edge (w1)
			(v.center) edge (w2)
			(h2.center) to (w2)
			(w1) edge [ultra thick, color=pathColor3] (x1)
			(w2) edge (x2)
			(x1) edge [ultra thick, color=pathColor3] (y1) edge [bend left=40pt] (z2)
			(x2) edge (y1) edge [bend right=40pt] (z1)
			(y1) edge [ultra thick, color=pathColor3] (y2)
			(y2) edge [ultra thick, color=pathColor3] (z1) edge (z2)
			(z1) edge [ultra thick, color=pathColor3] (a1)
			(z2) edge (a2)
			(a1) edge [-, ultra thick, color=pathColor3] (hp2) edge[-] (up)
			(a2) edge[-] (hp1) edge [-](vp);
	
		\draw[->,ultra thick, color=pathColor3] (h1.center) to (w1);

		\node[arrowVertex, rotate=-90] (h1) at (ch1) [label=left: $h_1$] {};
		\node[arrowVertex, rotate=90] (h2) at (ch2) [label=left: $h_2$] {};
		\node[arrowVertex, rotate=-90] (hp1) at (chp1) [label=right: $h'_1$] {};
		\node[arrowVertex, rotate=90] (hp2) at (chp2) [label=right: $h'_2$] {};
	\end{tikzpicture}	\caption[A vertex path using the ``wrong'' vertex output.]{%
		A vertex path entering a gadget $\gadget{e}$ via a vertex input~$u$ and leaving it via the ``wrong'' vertex output~$v'$ has to use at least the red marked arcs.
		In this situation, an auxiliary path entering $\gadget{e}$ via~$h_1$ must follow the green arcs.
	}
	\label{fig:adp:vertex_path_correct_vertex_output}
\end{figure}
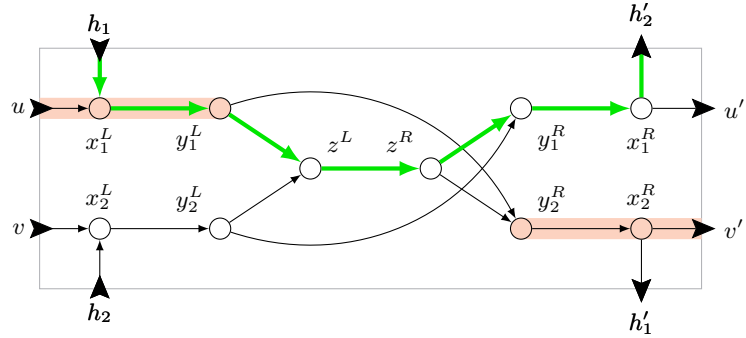

\begin{lemma} \label{lem:adp:no_gadget_with_two_vertex_paths}
	There is no gadget through which two vertex paths of~$\cQ$ pass.
\end{lemma}
\begin{proof}
	Suppose there is a gadget~$\gadget{e}$ that is passed through by two vertex paths.
	Denote these paths by $P_u$ and~$P_v$.
	Since they already have the arc $s v_V$ in common, they are disjoint in $\gadget{e}$.
	Hence, the path~$P_u$ entering $\gadget{e}$ via~$u$ has to use $y^R_2 x^R_2$ and the path~$P_v$ entering $\gadget{e}$ via~$v$ has to use $y^R_1 x^R_1$.

	Furthermore, there is also an auxiliary path passing through $\gadget{e}$ because otherwise we could replace one of the two vertex paths by an auxiliary path contradicting \cref{ass:adp_with_max_aux_paths}.
	Similarly to the proof of \cref{lem:adp:vertex_path_correct_vertex_output}, this auxiliary path has to be either an $h_1$-$h'_2$-path or an $h_2$-$h'_1$-path inside~$\gadget{e}$.
	By symmetry we assume without loss of generality that it is an $h_1$-$h'_2$-path.
	Hence, the situation is as depicted in \cref{fig:adp:two_vertex_paths_in_gadget} (the path $P_u$ can leave $\gadget{e}$ either via~$v'$ or via~$h'_1$).

	We now construct a new vertex path~$P$ that replaces $P_u$ and~$P_v$ in~$\cQ$.
	This path first uses $P_u$ until $y^L_1$ in $\gadget{e}$.
	From thereon it uses $(y^L_1, z^L, z^R, y^R_1)$ and then continues like~$P_v$.
	Since all vertex paths are disjoint after $v_V$, the new path~$P$ has only the arc $s v_V$ in common with any of the remaining vertex paths.
	Moreover, it has at most one arc in common with any auxiliary path outside of $\gadget{e}$ since this was already the case for $P_u$ and~$P_v$.
	We also replace the auxiliary $h_1$-$h'_2$-path by an auxiliary $h_1$-$h'_1$-path and an auxiliary $h_2$-$h'_2$-path.
	The resulting paths passing through $\gadget{e}$ are visualized in \cref{fig:adp:two_vertex_paths_in_gadget_after_modification}.

	After these modifications, $\cQ$ contains the same number of paths but the number of auxiliary paths increases by one.
	This contradicts \cref{ass:adp_with_max_aux_paths}.
	\qed
\end{proof}

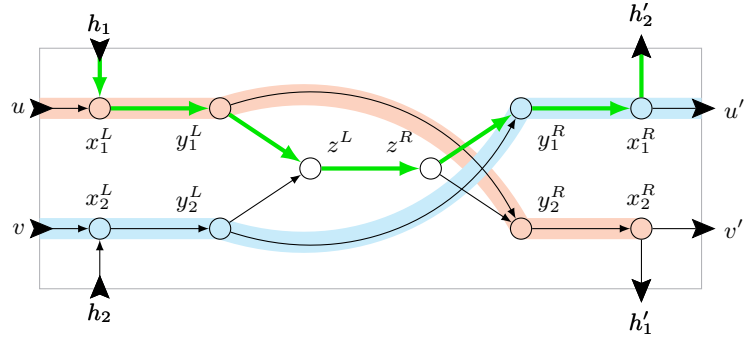
\begin{figure}[h]
	\centering
	\begin{tikzpicture}[scale=.8]
		\clip (-1,-2.75) rectangle (12,2.75);
		\draw[contourLine] (0,2) -- (11,2) -- (11,-2) -- (0,-2) -- cycle;
	
		\coordinate (cu) at (0,1);
		\coordinate (cv) at (0,-1);
		\coordinate (ch1) at (1,2);
		\coordinate (ch2) at (1,-2);
		\coordinate (cw1) at (1,1);
		\coordinate (cw2) at (1,-1);
		\coordinate (cx1) at (3,1);
		\coordinate (cx2) at (3,-1);
		\coordinate (cy1) at (4.5,0);
		\coordinate (cy2) at (6.5,0);
		\coordinate (cz1) at (8,1);
		\coordinate (cz2) at (8,-1);
		\coordinate (ca1) at (10,1);
		\coordinate (ca2) at (10,-1);
		\coordinate (cup) at (11,1);
		\coordinate (cvp) at (11,-1);
		\coordinate (chp2) at (10,2);
		\coordinate (chp1) at (10,-2);
	
		\draw[pathMarker1]
			(cu) -- (cw1) -- (cx1) to [bend left=40] (cz2);
		\draw[pathMarker1, line cap=round]
			(cz2) -- (ca2);
		\draw[pathMarker2]
			(cv) -- (cw2) -- (cx2) to [bend right=40]
			(cz1) -- (ca1) -- (cup);
	
		\node[arrowVertex] (u) at (cu) [label=left: $u$] {};
		\node[arrowVertex] (v) at (cv) [label=left: $v$] {};
		\node[arrowVertex, rotate=-90] (h1) at (ch1) [label=left: $h_1$] {};
		\node[arrowVertex, rotate=90] (h2) at (ch2) [label=left: $h_2$] {};
		\node[stdVertex] (w1) at (cw1) [label=below: $x^L_1$] {};
		\node[stdVertex] (w2) at (cw2) [label=above: $x^L_2$] {};
		\node[stdVertex] (x1) at (cx1) [label=below left: $y^L_1$] {};
		\node[stdVertex] (x2) at (cx2) [label=above left: $y^L_2$] {};
		\node[stdVertex] (y1) at (cy1) [label=above right: $z^L$] {};
		\node[stdVertex] (y2) at (cy2) [label=above left: $z^R$] {};
		\node[stdVertex] (z1) at (cz1) [label=below right: $y^R_1$] {};
		\node[stdVertex] (z2) at (cz2) [label=above right: $y^R_2$] {};
		\node[stdVertex] (a1) at (ca1) [label=below: $x^R_1$] {};
		\node[stdVertex] (a2) at (ca2) [label=above: $x^R_2$] {};
		\node[arrowVertex] (up) at (cup) [label=right: $u'$] {};
		\node[arrowVertex] (vp) at (cvp) [label=right: $v'$] {};
		\node[arrowVertex, rotate=-90] (hp1) at (chp1) [label=right: $h'_1$] {};
		\node[arrowVertex, rotate=90] (hp2) at (chp2) [label=right: $h'_2$] {};
	
		\draw[->] (u.center) edge (w1)
			(v.center) edge (w2)
			(h2.center) to (w2)
			(w1) edge [ultra thick, color=pathColor3] (x1)
			(w2) edge (x2)
			(x1) edge [ultra thick, color=pathColor3] (y1) edge [bend left=40pt] (z2)
			(x2) edge (y1) edge [bend right=40pt] (z1)
			(y1) edge [ultra thick, color=pathColor3] (y2)
			(y2) edge [ultra thick, color=pathColor3] (z1) edge (z2)
			(z1) edge [ultra thick, color=pathColor3] (a1)
			(z2) edge (a2)
			(a1) edge [-, ultra thick, color=pathColor3] (hp2) edge[-] (up)
			(a2) edge[-] (hp1) edge[-] (vp);
	
		\draw[->,ultra thick, color=pathColor3] (h1.center) to (w1);

		\node[arrowVertex, rotate=-90] (h1) at (ch1) [label=left: $h_1$] {};
		\node[arrowVertex, rotate=90] (h2) at (ch2) [label=left: $h_2$] {};
		\node[arrowVertex, rotate=-90] (hp1) at (chp1) [label=right: $h'_1$] {};
		\node[arrowVertex, rotate=90] (hp2) at (chp2) [label=right: $h'_2$] {};
	\end{tikzpicture}	\caption[Two vertex paths passing through a gadget.]{%
		The two vertex paths $P_u$ (marked red) and $P_v$ (marked blue) and an auxiliary path (drawn with green arcs) from \cref{lem:adp:no_gadget_with_two_vertex_paths} passing through a gadget.
	}
	\label{fig:adp:two_vertex_paths_in_gadget}
\end{figure}
\vspace{-5pt}
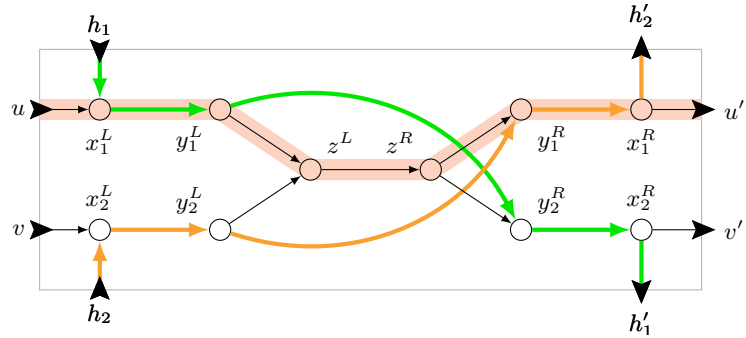
\begin{figure}[h]
	\centering
	\begin{tikzpicture}[scale=.8]
		\clip (-1,-2.75) rectangle (12,2.75);
		\draw[contourLine] (0,2) -- (11,2) -- (11,-2) -- (0,-2) -- cycle;
	
		\coordinate (cu) at (0,1);
		\coordinate (cv) at (0,-1);
		\coordinate (ch1) at (1,2);
		\coordinate (ch2) at (1,-2);
		\coordinate (cw1) at (1,1);
		\coordinate (cw2) at (1,-1);
		\coordinate (cx1) at (3,1);
		\coordinate (cx2) at (3,-1);
		\coordinate (cy1) at (4.5,0);
		\coordinate (cy2) at (6.5,0);
		\coordinate (cz1) at (8,1);
		\coordinate (cz2) at (8,-1);
		\coordinate (ca1) at (10,1);
		\coordinate (ca2) at (10,-1);
		\coordinate (cup) at (11,1);
		\coordinate (cvp) at (11,-1);
		\coordinate (chp2) at (10,2);
		\coordinate (chp1) at (10,-2);
	
		\draw[pathMarker1]
			(cu) -- (cw1) -- (cx1) -- (cy1) -- (cy2) --
			(cz1) -- (ca1) -- (cup);
	
		\node[arrowVertex] (u) at (cu) [label=left: $u$] {};
		\node[arrowVertex] (v) at (cv) [label=left: $v$] {};
		\node[arrowVertex, rotate=-90] (h1) at (ch1) [label=left: $h_1$] {};
		\node[arrowVertex, rotate=90] (h2) at (ch2) [label=left: $h_2$] {};
		\node[stdVertex] (w1) at (cw1) [label=below: $x^L_1$] {};
		\node[stdVertex] (w2) at (cw2) [label=above: $x^L_2$] {};
		\node[stdVertex] (x1) at (cx1) [label=below left: $y^L_1$] {};
		\node[stdVertex] (x2) at (cx2) [label=above left: $y^L_2$] {};
		\node[stdVertex] (y1) at (cy1) [label=above right: $z^L$] {};
		\node[stdVertex] (y2) at (cy2) [label=above left: $z^R$] {};
		\node[stdVertex] (z1) at (cz1) [label=below right: $y^R_1$] {};
		\node[stdVertex] (z2) at (cz2) [label=above right: $y^R_2$] {};
		\node[stdVertex] (a1) at (ca1) [label=below: $x^R_1$] {};
		\node[stdVertex] (a2) at (ca2) [label=above: $x^R_2$] {};
		\node[arrowVertex] (up) at (cup) [label=right: $u'$] {};
		\node[arrowVertex] (vp) at (cvp) [label=right: $v'$] {};
		\node[arrowVertex, rotate=-90] (hp1) at (chp1) [label=right: $h'_1$] {};
		\node[arrowVertex, rotate=90] (hp2) at (chp2) [label=right: $h'_2$] {};
	
		\draw[->] (u.center) edge (w1)
			(v.center) to (w2)
			(w1) edge [ultra thick, color=pathColor3] (x1)
			(w2) edge [ultra thick, color=pathColor4] (x2)
			(x1) edge (y1) edge [bend left=40pt, ultra thick, color=pathColor3] (z2)
			(x2) edge (y1) edge [bend right=40pt, ultra thick, color=pathColor4] (z1)
			(y1) edge (y2)
			(y2) edge (z1) edge (z2)
			(z1) edge [ultra thick, color=pathColor4] (a1)
			(z2) edge [ultra thick, color=pathColor3] (a2)
			(a1) edge [-, ultra thick, color=pathColor4] (hp2) edge[-] (up)
			(a2) edge [-, ultra thick, color=pathColor3] (hp1) edge[-] (vp);
		\draw[->,ultra thick, color=pathColor4] (h2.center) to (w2);
		\draw[->,ultra thick, color=pathColor3] (h1.center) to (w1);

		\node[arrowVertex, rotate=-90] (h1) at (ch1) [label=left: $h_1$] {};
		\node[arrowVertex, rotate=90] (h2) at (ch2) [label=left: $h_2$] {};
		\node[arrowVertex, rotate=-90] (hp1) at (chp1) [label=right: $h'_1$] {};
		\node[arrowVertex, rotate=90] (hp2) at (chp2) [label=right: $h'_2$] {};
	\end{tikzpicture}	\caption[Modified vertex paths passing through a gadget.]{%
		The result of the modifications in \cref{lem:adp:no_gadget_with_two_vertex_paths}.
		The two vertex paths and the auxiliary path from \cref{fig:adp:two_vertex_paths_in_gadget} are replaced by the vertex path marked red and the two auxiliary paths drawn in green and orange.
	}
	\label{fig:adp:two_vertex_paths_in_gadget_after_modification}
\end{figure}

Using \cref{lem:adp:aux_path_aux_output,lem:adp:vertex_path_correct_vertex_output,lem:adp:no_gadget_with_two_vertex_paths} we are now able to prove in \cref{lem:adp:number_almost_disjoint_paths_large_ind_set} the claim opposite to \cref{lem:adp:size_ind_set_many_almost_disjoint_paths} and thus complete the proof that \ADP is \NP-complete.

\begin{lemma} \label{lem:adp:number_almost_disjoint_paths_large_ind_set}
	Given $2 m + k$ almost disjoint $s$-$t$-paths in~$G$, there is an independent set $U \subseteq V_H$ in~$H$ of size $|U| = k$.
\end{lemma}
\begin{proof} ~
	We choose a set~$\cQ$ of $2 m + k$ almost disjoint $s$-$t$-paths in~$G$ that fulfills \cref{ass:adp_with_max_aux_paths}.
	By \cref{lem:adp:aux_path_aux_output} we also assume that every auxiliary path in~$\cQ$ passes through exactly one gadget.

	We define $U \subseteq V_H$ to be the set of vertices that are contained in a vertex path of~$\cQ$.
	We first prove that $U$ is an independent set in~$H$.

	\Cref{lem:adp:no_gadget_with_two_vertex_paths} implies that for every gadget there is at most one vertex path in~$\cQ$ that passes through this gadget.
	If this is the case, \cref{lem:adp:vertex_path_correct_vertex_output} states that this vertex path enters the gadget via a vertex input~$u$ and leaves it via the corresponding vertex output~$u'$.
	Thus, a vertex path starting with $(s, v_V, u)$ passes through a gadget $\gadget{e}$ if and only if $e \in \incidentedges[H]{u}$.
	Because no two vertex paths from~$\cQ$ pass through the same gadget, we obtain that $U$ is indeed an independent set.

	We complete the proof by showing that $U$ contains $k$ elements.
	Every gadget is used by at most one vertex path, see~\cref{lem:adp:no_gadget_with_two_vertex_paths}.
	Moreover, such a vertex path leaves every gadget via the correct vertex output, see~\cref{lem:adp:vertex_path_correct_vertex_output} again.
	Thus, we can additionally choose two auxiliary paths passing through every gadget.
	Furthermore, since $\cQ$ fulfills \cref{ass:adp_with_max_aux_paths}, it contains $2 m$ auxiliary paths.
	And since there can only be at most $2 m$ almost disjoint auxiliary paths, $\cQ$ contains exactly $k$ vertex paths.
	Because these are also almost disjoint, they contain distinct vertices $u \in V_H$ showing $|U| = k$.
	\qed
\end{proof}

Now that we have shown \ADP to be \NP-complete in general, we wish to remark that allowing paths to have more edges in common does not make the problem easier.
\begin{remark}
	\label{rem:adp:more_common_arcs}
	Let $l\in\NN$ with $l\geq 1$.
	Regard the following relaxation of \ADP:
	given a directed graph~$G = (V, A)$ together with two designated vertices $s, t \in V$ and a natural number $k \in \NN$.
	Are there $k$ $s$-$t$-paths such that any two of them have at most $l$ arcs in common?

	This problem is also \NP-hard, which can be seen by reducing \ADP to it.
	For this, we augment the graph $G$ of an \ADP instance by adding vertices $v_1,\ldots,v_{l-1}$ and arcs $v_1v_2,\ldots, v_{l-2}v_{l-1}$ and $v_{l-1}s$.
	In the resulting graph $G'$ any $v_1$-$t$-path first uses the $l-1$ newly added arcs.
	Therefore, $k$ almost disjoint $s$-$t$ paths in $G$ correspond to $k$ $v_1$-$t$ paths in $G'$ that have at most $l$ arcs in common and vice versa.
	This shows \NP-hardness.
\end{remark}

\section{Separating by Forbidden Pairs} \label{sec:sfp}
In this section, we prove \SFP to be \SigmaTwoP-complete yielding our third theorem.%
\sfpstwop*
\SFP is contained in $\SigmaTwoP = \NP^{\NP}$ since it can be solved by a non-deter\-ministic Turing machine that has access to an oracle for the \NP-complete \cite{GMO76} path avoiding forbidden pairs problem (cf.\ \cite[Remark~5.16]{AB09}).
In the remainder of this section we reduce the \SigmaTwoP-complete problem \SigmaTwoSat to \SFP proving its \SigmaTwoP-hardness.

\subsection*{The Problem \texorpdfstring{\SigmaTwoSat}{Sigma2-SAT}}

An instance of \SigmaTwoSat is given by a Boolean formula~$\varphi(x, y)$ depending on two types of variables.
The question is, whether an assignment of the $x$-variables exists such that~$\varphi(x, y)$ is \true for every assignment of the $y$-variables.
This problem, sometimes also denoted by \QSAT, is a standard \SigmaTwoP-complete problem, see \cite[Theorem~17.10]{Pap94} or \cite[Section~2.2.1]{dHaa19}.
We first introduce some notation that we use in order to deal with this problem.

\begin{notation} \label{not:sigma2sat}
	The Boolean formula~$\varphi = \varphi(x, y)$ depends on $n_x$ many $x$-variables $X = \{x_1, \dots, x_{n_x}\}$ and on $n_y$ many $y$-variables $Y = \{y_1, \dots, y_{n_y}\}$ whose union we denote by $Z = X \cup Y$.
	A \emph{truth assignment} $T \colon Z \to \{0, 1\}$ assigns a Boolean value to every variable.
	If we are only interested in the assignments of $x$-\ or $y$-variables, we write $T_X \colon X \to \{0, 1\}$ as well as $T_Y \colon Y \to \{0, 1\}$ and identify $T$ with $(T_X, T_Y)$, where $T_X = \restrict{T}{X}$ and $T_Y = \restrict{T}{Y}$.
\end{notation}

We say that the instance~$\varphi$ is \emph{satisfiable} if an $x$-variable assignment~$T_X$ exists such that $\varphi$ evaluates to \true for every $y$-variable assignment~$T_Y$.

\subsection*{Outline of the \texorpdfstring{\SigmaTwoP}{Sigma2P}-Hardness Proof}

To prove the hardness of \SFP, we construct a directed acyclic graph~$G$ for such a Boolean formula~$\varphi$.
For carefully chosen $k \in \NN$ we show that a source~$s$ and a target~$t$ in~$G$ can be separated by a set $\cA$ of $k$~forbidden pairs if and only if the \SigmaTwoSat instance~$\varphi$ is satisfiable.

In this graph~$G$, most separating pairs are predetermined.
Those that are not have essentially two options, which are used to encode assignments of the $x$-variables.
This means that an assignment~$T_X$ of the $x$-variables corresponds to a selection of forbidden pairs~$\cA$ and vice versa.
An assignment~$T_Y$ of the $y$-variables will correspond to $s$-$t$-paths in the graph that contain a pair from~$\cA$ if and only if the assignment $T = (T_X, T_Y)$ satisfies a clause.
From this we conclude that an assignment~$T_X$ exists such that $\varphi$ evaluates to \true for all assignments~$T_Y$ if and only if there exists a small set~$\cA$ such that every $s$-$t$-path contains a pair from~$\cA$.
However, the construction of the graph also generates $s$-$t$-paths that do not correspond to any $y$-variable assignment~$T_Y$.
To make the argumentation work, we have to enforce that all these paths contain forbidden pairs.

In the following, we start with non-restrictive assumptions about the Boolean formula~$\varphi$.
Thereafter, we introduce the different gadgets and concepts required for the final \SigmaTwoP-hardness proof.

\subsection*{Assumptions and Assignments}

Without loss of generality we may assume that the Boolean formula~$\varphi$ is given in \DNF[3], that is, in disjunctive normal form where each clause contains exactly three literals, see \cite[Section~2.2.1]{dHaa19}.
Hence, we can write $\varphi = C_1 \vee \dots \vee C_m$ as a disjunction of $m$ clauses where each clause is the conjunction of three literals.

\begin{assumption} \label{ass:sfp_hardness-formula-3dnf}
	The Boolean formula~$\varphi$ is given in \DNF[3].
\end{assumption}

Let us consider a clause consisting entirely of $x$-variables.
If it contains a variable $x_i$ and its negation $\overline{x_i}$, the clause can never be fulfilled and we can remove it.
Otherwise, we can satisfy this clause (and with it the entire formula~$\varphi$) solely by an appropriate $x$-variable assignment.
Hence, we might also assume that every clause contains at least one $y$-variable.

\begin{assumption} \label{ass:sfp_hardness-no_clause_only_x_vars}
	No clause of $\varphi$ consists entirely of $x$-variables.
\end{assumption}

\pagebreak[2]

Our last assumption is that no variable is contained in a single clause only.
This can be guaranteed, for example, by duplicating all clauses.

\begin{assumption} \label{ass:sfp_hardness-variable_in_at_least_2_clauses}
	Every variable is contained in at least two clauses of~$\varphi$.
\end{assumption}

\Cref{ass:sfp_hardness-formula-3dnf,ass:sfp_hardness-no_clause_only_x_vars} directly imply the following lemma.

\begin{lemma} \label{lem:sfp_hardness-1-2-3-y_vars_in_clause}
	Every clause contains either one, two, or three $y$-variables.
\end{lemma}

Before we describe the graph construction in detail, we introduce local as well as global $y$-variable assignments and define inconsistencies.

\begin{notation} \label{not:sigma2sat3dnf}
	The Boolean formula~$\varphi = C_1 \vee \dots \vee C_m$ is given as a disjunction of $m$ clauses, where every clause~$C_i = \ell_i^1 \wedge \ell_i^2 \wedge \ell_i^3$ is a conjunction of exactly three literals $\ell_i^j \in \{z, \overline{z} \setcolon z \in Z\}$.
	By~$Y(C) \subseteq Y$ we denote the set of $y$-variables that occur (negated or not) in a clause~$C$.
	We call an assignment of these variables a \emph{local ($y$-variable) assignment} and denote it by $T_{Y(C)} \colon Y(C) \to \{0, 1\}$.
	In the same spirit, we call $T_Y$ a \emph{global assignment}.
\end{notation}

\begin{definition} \label{def:sfp-hardness:inconsistency}
	Two local $y$-variable assignments $L = T_{Y(C)}$ and~$L' = T_{Y(C')}$ for distinct clauses $C$ and $C'$ are \emph{consistent} if they coincide on $Y(C) \cap Y(C')$.
	Otherwise, they are \emph{inconsistent} and the (unordered) pair $I = \{L, L'\}$ is an \emph{inconsistency}.
\end{definition}

We are now ready to start with our graph construction, which requires the introduction of several gadgets.

\subsection*{Graph Components}

\subsubsection{Inconsistency Gadgets.}
We start with the simplest gadget, the \emph{inconsistency gadget}.
It corresponds to an inconsistency and its only purpose is to enforce that a minimal separating set~$\cA$ contains a specific pair of arcs.
We use this gadget to ensure that paths not corresponding to a global $y$-variable assignment contain a forbidden pair.

Every inconsistency gadget is a directed acyclic graph as depicted in \cref{fig:sfp_inconsistency_gadget}.
It consists of an $s^I$-$t^I$-path with five arcs where the first, third, and last arc is replaced by two parallel arcs.

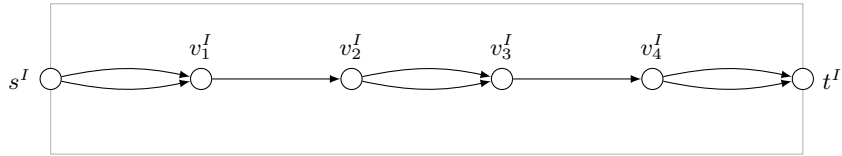
\begin{figure}
	\centering
	\begin{tikzpicture}
		\draw[contourLine] (0,1) -- (10,1) -- (10,-1) -- (0,-1) -- cycle;
	
		\node[stdVertex, fill=white] (s) at (0, 0) [label=left: $s^I$] {};
		\node[stdVertex, fill=white] (t) at (10, 0) [label=right: $t^I$] {};
		\foreach \i in {1,...,4}
		{
			\node[stdVertex] (v\i) at (2 * \i, 0) [label=above: $v^I_{\i}$] {};
		}
		\draw[stdArc] (s) edge [bend left=12] (v1)
	                      edge [bend right=12] (v1);
		\draw[stdArc] (v1) edge (v2);
		\draw[stdArc] (v2) edge [bend left=12] (v3)
		                   edge [bend right=12] (v3);
		\draw[stdArc] (v3) edge (v4);
		\draw[stdArc] (v4) edge [bend left=12] (t)
		                   edge [bend right=12] (t);
	\end{tikzpicture}	\caption[\SFP inconsistency gadget.]{%
		An inconsistency gadget corresponding to an inconsistency~$I$.
	}
	\label{fig:sfp_inconsistency_gadget}
\end{figure}

\begin{lemma} \label{lem:sfp_hardness-inconsistency_separation}
	The unique optimal solution to separate $s^I$ and~$t^I$ in an inconsistency gadget by forbidden pairs is $\cA_I = \{ \{ v^I_1 v^I_2, v^I_3 v^I_4 \} \}$.
\end{lemma}
\begin{proof}
	The set~$\cA_I$ separates $s^I$ and~$t^I$ and every separating set needs at least one pair.
	Thus, every optimal solution consists of a single forbidden pair.
	To prove the uniqueness, suppose there is an optimal solution whose pair contains one of two parallel arcs.
	In this case, a path using the other arc does not completely contain this pair, which yields a contradiction.
	\qed
\end{proof}

\subsubsection{Variable Gadgets.}
The \emph{variable gadgets} correspond to the $x$-variables in~$\varphi$.
Their purpose is to reflect a truth assignment~$T_X$ of these variables.
That is, there should be exactly two optimal sets of forbidden pairs in such a gadget:
one corresponding to setting the variable to \true and one for making it \false.
An illustration of such a gadget is given in \cref{fig:sfp_variable_gadget}.
We now describe its construction in more detail.
Thereafter, we explain what the two separating sets look like and prove that these are indeed the only two optimal solutions.

Basically, the variable gadget corresponding to a variable~$x_i$ consists of two vertices $s^i$ and~$t^i$ that are connected by several paths.
Similar to the inconsistency gadgets we double some arcs on these paths and we link them in a certain way.

The gadget contains an $s^i$-$t^i$-path for every occurrence of~$x_i$ in the formula~$\varphi$.
More precisely, the $j$-th occurrence corresponds to a path $(s^i, v^i_{j, 1}, \dots, v^i_{j, 7}, t^i)$ on which we replace the first, fourth, fifth, and last arc by two parallels.
Additionally, we add a path $(s^i, v^i_{0, 1}, v^i_{0, 2}, v^i_{0, 3}, v^i_{0, 5}, v^i_{0, 6}, v^i_{0, 7}, t^i)$, which is not associated with any occurrence.
On this path we replace the first, fourth, and last arc by two parallel arcs.
Furthermore, we introduce the arcs $v^i_{0, 3} v^i_{j, 4}$ and $v^i_{j, 4} v^i_{0, 5}$ between these paths.

\begin{figure}[h]
	\centering
	\begin{tikzpicture}[scale=.9]
		\draw[contourLine] (0, 3) -- (12, 3) -- (12, -3) -- (0, -3) -- cycle;
	
		\node[stdVertex, fill=white] (s) at (0, 0) [label=left: $s^i$] {};
		\node[stdVertex, fill=white] (t) at (12, 0) [label=right: $t^i$] {};
		
		\foreach \j in {1,...,7}
		{
			\pgfmathsetmacro\jp{int(\j - 1)}
			\foreach \i in {0,1,2,4}
			{
				\def\xshift{0pt}
				\def\yshift{-2pt}
				\def\ilabel{\i}
				\def\labelpos{above}
				\ifthenelse{\i = 1}{%
					\ifthenelse{\j = 3 \OR \j = 7}{%
						\def\xshift{-7pt}
					}{%
						\ifthenelse{\j = 1 \OR \j = 5}{%
							\def\xshift{7pt}
						}{}
					}
				}{}
				\ifthenelse{\i = 2 \OR \i = 4}{%
					\def\labelpos{below}
					\def\yshift{0pt}
					 \ifthenelse{\i = 4}{
						\def\ilabel{q}
					}{}
				}{}
				\ifthenelse{\i = 0 \AND \j = 4}{}{
					\node[stdVertex] (v\i-\j) at (1.5 * \j, 2 - \i) [label={\labelpos, xshift=\xshift, yshift=\yshift}: $v^i_{\ilabel,\j}$] {};
				}
				\ifthenelse{\j > 1}{%
					\ifthenelse{\j = 4 \OR \j = 5}{%
						\ifthenelse{\i = 0}{}{%
							\draw[stdArc] (v\i-\jp) edge [bend left=7] (v\i-\j)
			                              (v\i-\jp) edge [bend right=7] (v\i-\j);
			            }
					}{%
						\draw[stdArc] (v\i-\jp) edge (v\i-\j);
					}
				}{}
				\ifthenelse{\j = 5 \AND \i > 0}{
					\draw[stdArc] (v0-3) edge (v\i-4);
					\draw[stdArc] (v\i-4) edge (v0-5);
				}{}
				\ifthenelse{\j = 1}{
					\draw[stdArc] (s) edge [bend left=5] (v\i-1)
					                  edge [bend right=5] (v\i-1);
				}{}
				\ifthenelse{\j = 7}{
					\draw[stdArc] (v\i-7) edge [bend left=5] (t)
				                          edge [bend right=5] (t);
				}{}
			}
			\node at ($(v2-\j)!0.6!(v4-\j)$) {$\vdots$};
		}
		
		\draw[stdArc] (v0-3) edge [bend left=5] (v0-5);
		\draw[stdArc] (v0-3) edge [bend right=5] (v0-5);
	\end{tikzpicture}	\caption[\SFP variable gadget.]{%
		A variable gadget corresponding to variable~$x_i$.
		We use $q = q_i$ for the number of occurrences of~$x_i$ (including negated literals) in formula~$\varphi$.
	}
	\label{fig:sfp_variable_gadget}
\end{figure}
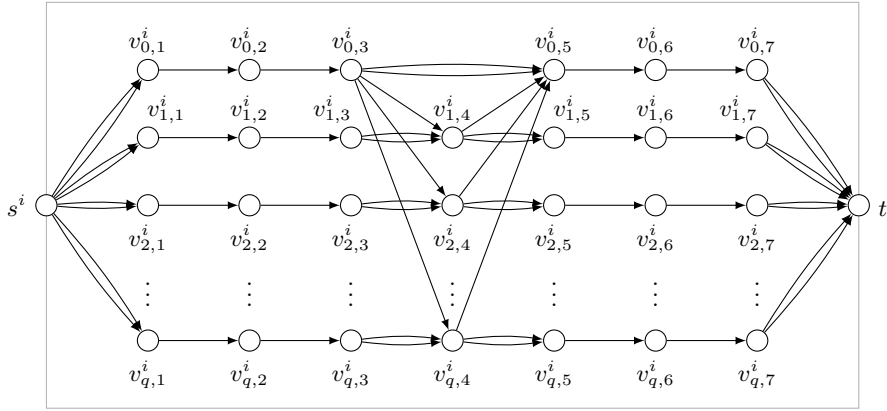

Let $q_i$ denote the number of occurrences of variable~$x_i$ in the formula~$\varphi$.
As there are, by construction, $q_i + 1$ arc-disjoint $s^i$-$t^i$-paths in this gadget, an optimal set of forbidden pairs separating $s^i$ and~$t^i$ must contain at least $q_i + 1$ pairs.
Thus, the two separating sets $\cA_i = \{ \{ v^i_{j, 1} v^i_{j, 2}, v^i_{j, 2} v^i_{j, 3} \} \setcolon j = 0, \dots, q_i \}$ and $\overline{\cA_i} = \{ \{ v^i_{j, 5} v^i_{j, 6}, v^i_{j, 6} v^i_{j, 7} \} \setcolon j = 0, \dots, q_i \}$ are optimal.
The following lemma shows that these are, in fact, the only two optimal sets of forbidden pairs.
We identify choosing the separating set $\cA_i$ with setting $x_i$ to \true and choosing $\overline{\cA_i}$ with setting $x_i$ to \false.

\begin{lemma} \label{lem:sfp_hardness-separating_sets_variable_gadget}
	The sets $\cA_i$ and~$\overline{\cA_i}$ are the only optimal sets of forbidden pairs separating $s^i$ and~$t^i$ in the variable gadget corresponding to variable~$x_i$.
\end{lemma}
\begin{proof}
	As argued above, an optimal separating set contains exactly~$q_i + 1$ pairs, thus $\cA_i$ and~$\overline{\cA_i}$ are optimal separating sets.
	It remains to prove their uniqueness.

	Similarly to the proof of \cref{lem:sfp_hardness-inconsistency_separation} we can show that no forbidden pair of an optimal solution uses one of two parallel arcs:
	otherwise, there are still $q_i + 1$ disjoint $s^i$-$t^i$-paths, none of which completely contains this pair.
	With the same argumentation it follows that none of the arcs $v^i_{0, 3} v^i_{j, 4}$ and $v^i_{j, 4} v^i_{0, 5}$ between these paths is contained in a forbidden pair of an optimal solution.

	Thus, all forbidden pairs are composed of arcs of the form $v^i_{j, 1} v^i_{j, 2}$, $v^i_{j, 2} v^i_{j, 3}$, $v^i_{j, 5} v^i_{j, 6}$, and $v^i_{j, 6} v^i_{j, 7}$.
	For $j \in \{1, \dots, q_i\}$ we consider the four different paths
	\begin{align*}
	& (s^i, v^i_{0, 1}, \dots, v^i_{0, 7}, t^i),
	&& (s^i, v^i_{0, 1}, v^i_{0, 2}, v^i_{0, 3}, v^i_{j, 4}, \dots, v^i_{j, 7}, t^i),\\
	& (s^i, v^i_{j, 1}, \dots, v^i_{j, 7}, t^i), \text{ and}
	&& (s^i, v^i_{j, 1}, \dots, v^i_{j, 4}, v^i_{0, 5}, v^i_{0, 6}, v^i_{0, 7}, t^i).
	\end{align*}
	An optimal solution has to separate these four paths with only two forbidden pairs as there are $q_i - 1$ disjoint paths in the remaining gadget.
	This, however, is only possible if either the pairs $\{v^i_{0, 1} v^i_{0, 2}, v^i_{0, 2} v^i_{0, 3}\}$ and $\{v^i_{j, 1} v^i_{j, 2}, v^i_{j, 2} v^i_{j, 3}\}$ or the pairs $\{v^i_{0, 5} v^i_{0, 6}, v^i_{0, 6} v^i_{0, 7}\}$ and $\{v^i_{j, 5} v^i_{j, 6}, v^i_{j, 6} v^i_{j, 7}\}$ are chosen.
	Since this holds for all $j \in \{1, \dots, q_i\}$, the claim follows.
	\qed
\end{proof}

\subsubsection{Formula Gadget.}
The \emph{formula gadget} consists of clause assignment units, which we describe later, that are ordered in a layered structure.
For now it suffices to know that they have one input vertex and one output vertex which we use to connect them.
By \cref{lem:sfp_hardness-1-2-3-y_vars_in_clause}, every clause~$C$ of~$\varphi$, contains $\ell \in \{1, 2, 3\}$ many $y$-variables.
For every of the $2^\ell$ possible local $y$-variable assignments~$L = T_{Y(C)}$ for $C$ we introduce one such clause assignment unit.
We denote its input vertex by~$s^L$ and its output vertex by $t^L$.
This yields either two, four, or eight clause assignment units for each clause.

The $i$-th layer of the formula gadget consists of all clause assignment units corresponding to the $i$-th clause of~$\varphi$.
A source~$s^0$ is connected to the input~$s^L$ of every clause assignment unit corresponding to a local assignment~$L = T_{Y(C_1)}$ of the first clause.
In addition, we connect the clause assignment units of successive clauses in the formula gadget by complete bipartite graphs.
Finally, we connect every output~$t^L$ of a unit corresponding to the last clause~$C_m$ with the target~$t^0$.
The structure of the formula gadget is visualized in \cref{fig:sfp_formula_gadget}.

\begin{figure}
	\centering
	\newcommand{\drawAssignmentUnit}[4]{%
		\draw[contourLine] (#1, #2) -- ++(0, .5) -- ++(2, 0) -- ++(0, -1) -- ++(-2, 0) -- cycle;
		
		\node[smallVertex, fill=white] (s-#3) at (#1, #2) {};
		\node[smallVertex, fill=white] (t-#3) at ($(s-#3) + (2, 0)$) {};
		\node at ($(s-#3)!0.5!(t-#3)$) {#4};
	}
	
	\begin{tikzpicture}[scale=.9]
		\draw[contourLine] (0, 2.5) -- (11, 2.5) -- (11, -2.5) -- (0, -2.5) -- cycle;
		
		\node[stdVertex, fill=white] (s) at (0, 0) [label=left: $s^0$] {};
		\node[stdVertex, fill=white] (t) at (11, 0) [label=right: $t^0$] {};
		
		\drawAssignmentUnit{1}{1}{1-1}{$T_{Y(C_1)}^1$}
		\drawAssignmentUnit{1}{-1}{1-2}{$T_{Y(C_1)}^2$}
		
		\drawAssignmentUnit{4}{1.8}{2-1}{$T_{Y(C_2)}^1$}
		\drawAssignmentUnit{4}{.6}{2-2}{$T_{Y(C_2)}^2$}
		\drawAssignmentUnit{4}{-.6}{2-3}{$T_{Y(C_2)}^3$}
		\drawAssignmentUnit{4}{-1.8}{2-4}{$T_{Y(C_2)}^4$}
		
		\node at (7, 0) {$\cdots$};
		
		\drawAssignmentUnit{8}{1}{3-1}{$T_{Y(C_m)}^1$}
		\drawAssignmentUnit{8}{-1}{3-2}{$T_{Y(C_m)}^2$}
		
		\foreach \j in {1,2}
		{
			\draw[stdArc] (s) edge (s-1-\j);
			\draw[stdArc] (t-3-\j) edge (t);
			\foreach \i in {1,...,4}
			{
				\draw[stdArc] (t-1-\j) edge (s-2-\i);
			}
		}
	\end{tikzpicture}	\vspace{-7pt}
	\caption[\SFP formula gadget.]{%
		The formula gadget for the  formula~$\varphi = C_1 \vee C_2 \vee \dots \vee C_m$ in 3-DNF.
		To distinguish the different possible local assignments of a clause~$C$ we enumerate them $T^1_{Y(C)}$, $T^2_{Y(C)}$, and so on.
	}
	\label{fig:sfp_formula_gadget}
\end{figure}
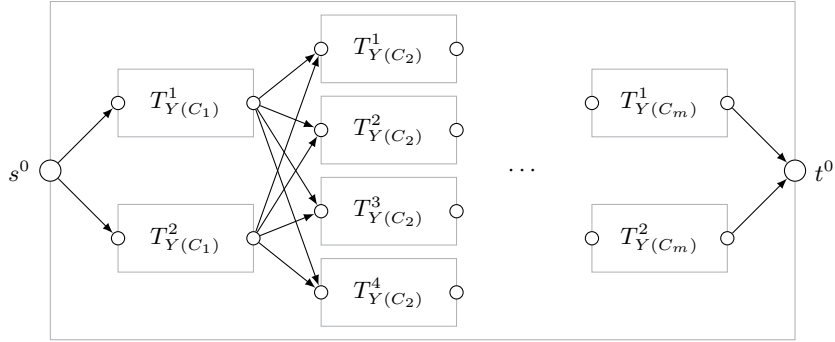

Most clause assignment units provide paths from their input to their output vertex.
Therefore, $s^0$-$t^0$-paths through the formula gadget pass through exactly one clause assignment unit of every layer.
This way, every such path selects a local $y$-variable assignment for every clause.
If these are consistent, that is, if every $y$-variable is assigned the same truth value in each local assignment of a clause that contains it, they can be combined to a global $y$-variable assignment.
The other way around, we can also associate an assignment~$T_Y$ with an $s^0$-$t^0$-path which uses in every layer the clause assignment unit corresponding to $T_{Y(C)} = \restrict{T_Y}{Y(C)}$ for the respective clause~$C$.

Thus, the paths through the formula gadget are linked with the global $y$-variable assignments.
Our goal is to ensure that any such path contains a forbidden pair if and only if the associated assignment satisfies the formula~$\varphi$ (in conjunction with the $x$-variable assignment).
For this, the variable gadgets will play an important role.
However, we also have to take those paths into consideration that do not correspond to consistent $y$-variable assignments.
In order to ensure that these paths contain forbidden pairs as well, we will make use of the inconsistency gadgets.

\subsubsection{Typification.}
All the gadgets introduced until now need to be part of $s$-$t$-paths in the final graph.
Since it will be possible to travel between gadgets later, new paths arise.
In particular, we obtain ``mixed'' paths that start at the source of one gadget and end at the terminal of another.
In order to keep these mixed paths in check when we finally put these pieces together we need the concept of \emph{typification}.

To explain the idea of typification, we start with a small example.
Given are two disjoint graphs $G_1$ and~$G_2$.
In each graph $G_i$ we want to separate a source~$s^i$ and a target~$t^i$ by forbidden pairs.
However, we want to combine these two graphs to a single graph~$G$ without affecting the optimal choice of forbidden pairs, that is, we still only want to select pairs in $G_1$ and~$G_2$.
Simply adding a source~$s$, a target~$t$, and connecting these with arcs $s s^1$, $s s^2$, $t^1 t$, and $t^2 t$ does not suffice as the combined instance can always be separated by the two forbidden pairs $\{s s^1, t^1 t\}$ and $\{s s^2, t^2 t\}$.
But if we know that $p - 1$ pairs are sufficient to separate $G_i$ for $i \in \{1, 2\}$, we can replace each of the four additional arcs by a bunch of $p$ parallel arcs.
In other words: if $k_i$ pairs are sufficient to separate~$G_i$, we can choose any $p > \max \{k_1, k_2\}$.
Therefore, an optimal solution in the combined instance only uses arcs that are contained within the subgraphs $G_1$ and~$G_2$.

If we also add $p + 1$ parallel arcs from $s^1$ to $t^2$ as well as from $s^2$ to $t^1$, we have to choose forbidden pairs separating all paths $(s, s^1, t^2, t)$ and all paths $(s, s^2, t^1, t)$.
These paths only consist of additional arcs not contained in the original graphs $G_1$ and~$G_2$.
The unique optimal solution to separate these paths is to choose the $2 p^2$~forbidden pairs that combine an arc~$s s^1$ with an arc~$t^2 t$ and an arc~$s s^2$ with an arc~$t^1 t$.
Thus, we can separate~$G_1$ by $k_1$ forbidden pairs and $G_2$ by $k_2$ forbidden pairs if and only if we can separate~$G$ by $2 p^2 + k_1 + k_2$ forbidden pairs.
This situation is visualized in \cref{fig:sfp_typification}.

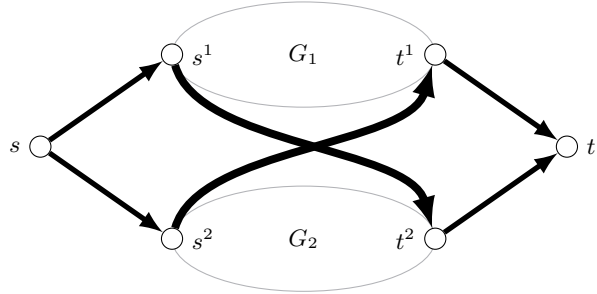
\begin{figure}
	\centering
	\begin{tikzpicture}[scale=.7]
		\node[stdVertex] (s) at (0, 0) [label=left: $s$] {};
		\node[stdVertex] (t) at (10, 0) [label=right: $t$] {};
		
		\node (g1) at (5, 1.75) {$G_1$};
		\node (g2) at (5, -1.75) {$G_2$};
		
		\draw[contourLine] (g1) circle [x radius=2.5, y radius=1];
		\draw[contourLine] (g2) circle [x radius=2.5, y radius=1];
		
		\node[stdVertex, fill=white] (s1) at (2.5, 1.75) [label=right: $s^1$] {};
		\node[stdVertex, fill=white] (t1) at (7.5, 1.75) [label=left: $t^1$] {};
		\node[stdVertex, fill=white] (s2) at (2.5, -1.75) [label=right: $s^2$] {};
		\node[stdVertex, fill=white] (t2) at (7.5, -1.75) [label=left: $t^2$] {};
		
		\draw[->, line width=2pt] (s) edge (s1) edge (s2);
		\draw[->, line width=2pt] (t1) edge (t);
		\draw[->, line width=2pt] (t2) edge (t);
		
		\draw[->, line width=3pt] (s1) .. controls (3, 0) and (7, 0) .. (t2);
		\draw[->, line width=3pt] (s2) .. controls (3, 0) and (7, 0) .. (t1);
	\end{tikzpicture}	\caption[An example for typification.]{
		An exemplary typification construction.
		The bold arcs $s s^1$, $s s^2$, $t^1 t$, and $t^2 t$ represent bunches of $p$ parallel arcs.
		The even thicker arcs $s^1 t^2$ and $s^2 t^1$ represent bunches of $p + 1$ parallel arcs.
	}
	\label{fig:sfp_typification}
\end{figure}

The reason we introduce these additional arcs is because they help us weed out mixed paths:
if we allow arcs between $G_1$ and~$G_2$ in~$G$, then it becomes possible to obtain $s$-$t$-paths containing $s^i$ and~$t^j$ for $i \neq j$.
By adding the additional ``diagonal'' arcs $s^1 t^2$ and $s^2 t^1$ we enforce the choice of all $2 p^2$ pairs $\{s s^i, t^j t\}$ for $i \neq j$ and, thus, ensure that these mixed paths are already saturated with at least one pair.
This just leaves paths that start with~$s s^i$ and end with~$t^i t$ for all possible indices~$i$.
We only have to examine whether all paths of these two \emph{types} contain a forbidden pair or not.
Note that this does include paths that are not solely part of a subgraph~$G_i$, as they can leave and return, but it does reduce the potential paths without a forbidden pair immensely.

This construction can be generalized to more than only two types.
For $q$~subgraphs $G_1, \dots, G_q$ with sources~$s^i$ and targets~$t^i$, $i \in \{1, \dots, q\}$, we can add $p$~parallel arcs from $s$ to every source~$s^i$ and from every target~$t^i$ to~$t$.
Additionally, we add $p + 1$ parallel arcs~$s^i t^j$ for all $i, j \in \{1, \dots, q\}$ with $i \neq j$.
Every optimal solution has to use the $p^2$ forbidden pairs of arcs $\{s s^i, t^j t\}$ of different types~$i \neq j$.
Thus, every optimal solution has $p^2 q (q - 1)$ forbidden pairs and, additionally, the pairs required to separate all paths of the $q$ different types (all paths using $s s^i$ and $t^i t$ for some~$i$).
We intend to use this to give all inconsistency gadgets, all variable gadgets, as well as the formula gadget their own type.

\subsubsection{Clause Assignment Units and Graph Construction.}
To construct the graph corresponding to formula~$\varphi$ we use the formula gadget, a variable gadget for every $x$-variable, and several inconsistency gadgets.
More precisely, for every pair of clause assignment units (within the formula gadget) that corresponds to incompatible assignments we introduce one such inconsistency gadget.
All these gadgets are combined into the graph~$G$ as explained in the typification section.

Let us describe the graph construction in detail.
That is, we finally have to specify what the clause assignment units look like and how these are connected to the other gadgets.
Recall that the formula gadget contains a clause assignment unit for every clause~$C$ and every possible assignment~$L = T_{Y(C)}$ of Boolean values to the $y$-variables contained in~$C$.
As already stated in the formula gadget section, these are $2^\ell$~clause assignment units for a clause with $\ell$~many $y$-variables.

A \emph{clause assignment unit} corresponding to a $y$-variable assignment~$L$ of a clause~$C$ contains exactly three vertices: $s^L$, $v^L$, and $t^L$.
Note that we use $s^L$ and~$t^L$ in order to connect the clause assignment units in the formula gadget as described above.
The vertices $v^L$ and~$t^L$ are connected by an arc~$v^L t^L$ if and only if $C$ contains at least one $y$-literal that evaluates to \false with the $y$-variable assignment~$L$.
Hence, this arc is present in all but one clause assignment unit corresponding to a clause~$C$ since there is only exactly one assignment~$T_{Y(C)}$ that satisfies all $y$-literals in~$C$.

These are all components within a clause assignment unit.
In particular, the clause assignment units are not connected and, thus, neither is the formula gadget.
The following modifications only add some arcs between different gadgets.
These are illustrated by dashed arcs in \cref{fig:sfp_clause_assignment_unit,fig:sfp_variable_gadget_with_connections,fig:sfp_inconsistency_gadget_with_connections}.

In addition to the (potentially non-existing) arc~$v^L t^L$, we add another path from $v^L$ to~$t^L$ for every $x$-literal contained in~$C$.
The path of a literal corresponding to variable~$x_i$ passes through the variable gadget of~$x_i$.
If the occurrence of~$x_i$ in~$C$ is the $j$-th occurrence in~$\varphi$ in total, this path uses either the arcs $v^i_{j, 1} v^i_{j, 2}$ and $v^i_{j, 2} v^i_{j, 3}$ (if~$C$ contains the literal~$x_i$) or the arcs $v^i_{j, 5} v^i_{j, 6}$ and $v^i_{j, 6} v^i_{j, 7}$ (if~$C$ contains the literal~$\overline{x}_i$).
In the former case we add the inter-gadget arcs $v^L v^i_{j, 1}$ and~$v^i_{j, 3} t^L$ and in the latter case we add $v^L v^i_{j, 5}$ as well as~$v^i_{j, 7} t^L$.
These connecting arcs are indicated by dashed orange arcs in \cref{fig:sfp_clause_assignment_unit,fig:sfp_variable_gadget_with_connections}.

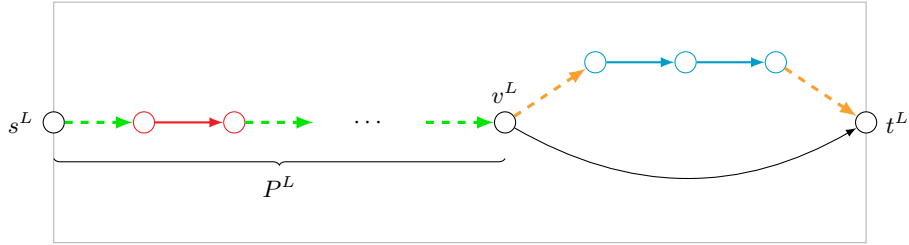
\begin{figure}[p]
	\centering
	\begin{tikzpicture}[scale=.8]
		\draw[contourLine] (0, 2) -- (13.5, 2) -- (13.5, -2) -- (0, -2) -- cycle;
	
		\node[stdVertex, fill=white] (s) at (0, 0) [label=left: $s^L$] {};
		\node[stdVertex, color=pathColor1] (v0) at (1.5, 0) {};
		\node[stdVertex, color=pathColor1] (v1) at (3, 0) {};
		\node[stdVertexPhantom] (v2) at (4.5, 0) {};
		\node[stdVertexPhantom] (v3) at (6, 0) {};
		\node at ($(v2)!0.5!(v3)$) {$\cdots$};
		\node[stdVertex] (v4) at (7.5, 0) [label=above: $v^L$] {};
		\node[stdVertex, color=pathColor2] (v5) at (9, 1) {};
		\node[stdVertex, color=pathColor2] (v6) at (10.5, 1) {};
		\node[stdVertex, color=pathColor2] (v7) at (12, 1) {};
		\node[stdVertex, fill=white] (t) at (13.5, 0) [label=right: $t^L$] {};
		
		\draw[->, dashed, very thick, color=pathColor3] (s) edge (v0);
		\draw[stdArc, thick, color=pathColor1] (v0) edge (v1);
		\draw[->, dashed, very thick, color=pathColor3] (v1) edge (v2);
		\draw[->, dashed, very thick, color=pathColor3] (v3) edge (v4);
		\draw[->, dashed, very thick, color=pathColor4] (v4) edge (v5);
		\draw[stdArc, thick, color=pathColor2] (v5) edge (v6);
		\draw[stdArc, thick, color=pathColor2] (v6) edge (v7);
		\draw[->, dashed, very thick, color=pathColor4] (v7) edge (t);
		\draw[stdArc] (v4) edge[bend right] (t);
		
		\draw[decorate, decoration={brace}] (7.5, -.6) -- node [label=below: $P^L$] {} (0, -.6);
	\end{tikzpicture}	\caption[\SFP clause assignment unit.]{%
		A clause assignment unit corresponding to a local $y$-variable assignment~$L$ for clause~$C$ containing a single $x$-variable.
		The assignment~$L$ does not fulfill all $y$-literals of~$C$ as the arc~$v^L t^L$ is present.
		The colored, solid arcs are contained in other gadgets and the dashed arcs connect these, see also \cref{fig:sfp_variable_gadget_with_connections,fig:sfp_inconsistency_gadget_with_connections}.
	}
	\label{fig:sfp_clause_assignment_unit}
\end{figure}

\begin{figure}[p]
	\centering
	\begin{tikzpicture}[scale=.9]
		\draw[contourLine] (0, 3) -- (12, 3) -- (12, -3) -- (0, -3) -- cycle;
	
		\node[stdVertex, fill=white] (s) at (0, 0) [label=left: $s^i$] {};
		\node[stdVertex, fill=white] (t) at (12, 0) [label=right: $t^i$] {};
		
		\foreach \j in {1,...,7}
		{
			\pgfmathsetmacro\jp{int(\j - 1)}
			\foreach \i in {0,1,2,4}
			{
				\def\xshift{0pt}
				\def\yshift{-2pt}
				\def\ilabel{\i}
				\def\labelpos{above}
				\def\col{black}
				\def\thickness{thin}
				\ifthenelse{\i = 1}{%
					\ifthenelse{\j = 3 \OR \j = 7}{%
						\def\xshift{-7pt}
					}{%
						\ifthenelse{\j = 1 \OR \j = 5}{%
							\def\xshift{7pt}
						}{}
					}
				}{}
				\ifthenelse{\i = 2 \OR \i = 4}{%
					\def\labelpos{below}
					\def\yshift{0pt}
					 \ifthenelse{\i = 4}{
						\def\ilabel{q}
					}{}
				}{}
				\ifthenelse{\j > 0 \AND \j < 4}{
					\ifthenelse{\i = 1 \OR \i = 4}{
						\def\col{pathColor2}
						\def\thickness{thick}
					}{}
				}{}
				\ifthenelse{\j > 4 \AND \j < 8 \AND \i = 2}{
					\def\col{pathColor2}
					\def\thickness{thick}
				}{}
				\ifthenelse{\i = 0 \AND \j = 4}{}{
					\node[stdVertex, color=\col] (v\i-\j) at (1.5 * \j, 2 - \i) [label={\labelpos, xshift=\xshift, yshift=\yshift}: $v^i_{\ilabel,\j}$] {};
				}
				\ifthenelse{\j > 1}{%
					\ifthenelse{\j = 4 \OR \j = 5}{%
						\ifthenelse{\i = 0}{}{%
							\draw[stdArc] (v\i-\jp) edge [bend left=7] (v\i-\j)
			                              (v\i-\jp) edge [bend right=7] (v\i-\j);
			            }
					}{%
						\draw[stdArc] (v\i-\jp) edge[color=\col, \thickness] (v\i-\j);
					}
				}{}
				\ifthenelse{\j = 5 \AND \i > 0}{
					\draw[stdArc] (v0-3) edge (v\i-4);
					\draw[stdArc] (v\i-4) edge (v0-5);
				}{}
				\ifthenelse{\j = 1}{
					\draw[stdArc] (s) edge [bend left=5] (v\i-1)
					                  edge [bend right=5] (v\i-1);
				}{}
				\ifthenelse{\j = 7}{
					\draw[stdArc] (v\i-7) edge [bend left=5] (t)
				                          edge [bend right=5] (t);
				}{}
			}
			\node at ($(v2-\j)!0.6!(v4-\j)$) {$\vdots$};
		}
		
		\draw[stdArc] (v0-3) edge [bend left=5] (v0-5);
		\draw[stdArc] (v0-3) edge [bend right=5] (v0-5);
		
		\node[stdVertexPhantom] (in1) at (1.1, .25) {};
		\node[stdVertexPhantom] (out1) at (4.9, .25) {};
		\node[stdVertexPhantom] (in2) at (6.7, -.75) {};
		\node[stdVertexPhantom] (out2) at (11.3, -.75) {};
		\node[stdVertexPhantom] (in3) at (.8, -2.75) {};
		\node[stdVertexPhantom] (out3) at (5.2, -2.75) {};
		
		\draw[->, dashed, very thick, color=pathColor4] (in1) edge (v1-1);
		\draw[->, dashed, very thick, color=pathColor4] (v1-3) edge (out1);
		\draw[->, dashed, very thick, color=pathColor4] (in2) edge (v2-5);
		\draw[->, dashed, very thick, color=pathColor4] (v2-7) edge (out2);
		\draw[->, dashed, very thick, color=pathColor4] (in3) edge (v4-1);
		\draw[->, dashed, very thick, color=pathColor4] (v4-3) edge (out3);
	\end{tikzpicture}	\caption[Variable gadget with connections.]{%
		A variable gadget (as in \cref{fig:sfp_variable_gadget}) with the connections to clause assignment units.
		The blue arcs and the dashed orange arcs correspond to these from \cref{fig:sfp_clause_assignment_unit}.
		In this example, the first and last occurrence of the corresponding $x$-variable occurs non-negated and the second occurrence is negated.
	}
	\label{fig:sfp_variable_gadget_with_connections}
\end{figure}
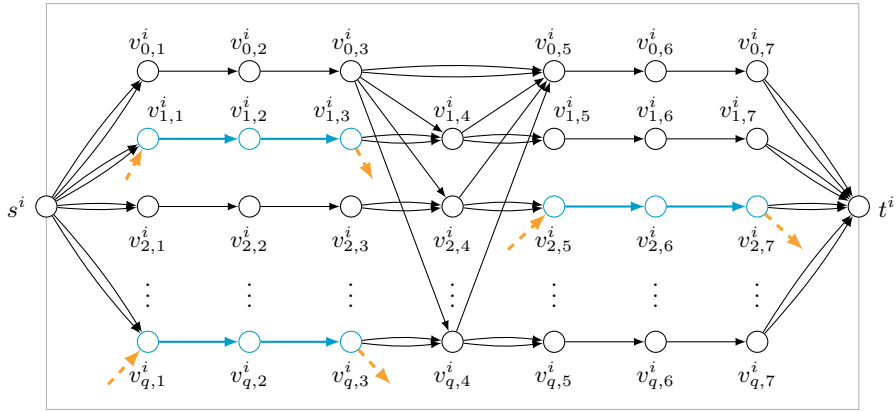

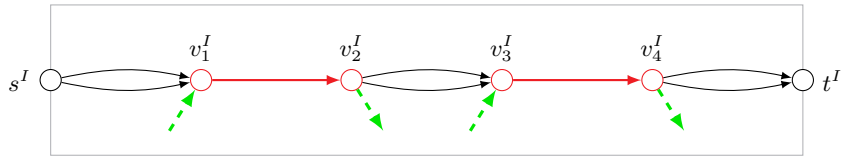
\begin{figure}[p]
	\centering
	\begin{tikzpicture}
		\draw[contourLine] (0,1) -- (10,1) -- (10,-1) -- (0,-1) -- cycle;
	
		\node[stdVertex, fill=white] (s) at (0, 0) [label=left: $s^I$] {};
		\node[stdVertex, fill=white] (t) at (10, 0) [label=right: $t^I$] {};
		\foreach \i in {1,...,4}
		{
			\node[stdVertex, color=pathColor1] (v\i) at (2 * \i, 0) [label=above: $v^I_{\i}$] {};
		}
	
		\node[stdVertexPhantom] (in1) at (1.5, -.8) {};
		\node[stdVertexPhantom] (out1) at (4.5, -.8) {};
		\node[stdVertexPhantom] (in2) at (5.5, -.8) {};
		\node[stdVertexPhantom] (out2) at (8.5, -.8) {};
	
		\draw[stdArc] (s) edge [bend left=12] (v1)
	                      edge [bend right=12] (v1);
		\draw[stdArc, thick, color=pathColor1] (v1) edge (v2);
		\draw[stdArc] (v2) edge [bend left=12] (v3)
		                   edge [bend right=12] (v3);
		\draw[stdArc, thick, color=pathColor1] (v3) edge (v4);
		\draw[stdArc] (v4) edge [bend left=12] (t)
		                   edge [bend right=12] (t);
		                   
		\draw[->, dashed, very thick, color=pathColor3] (in1) edge (v1);
		\draw[->, dashed, very thick, color=pathColor3] (v2) edge (out1);
		\draw[->, dashed, very thick, color=pathColor3] (in2) edge (v3);
		\draw[->, dashed, very thick, color=pathColor3] (v4) edge (out2);
	\end{tikzpicture}	\caption[Inconsistency gadget with connections.]{%
		An inconsistency gadget (as in \cref{fig:sfp_inconsistency_gadget}) with the connections to clause assignment units or other inconsistency gadgets.
		The red arcs and the green dashed arcs correspond to these from \cref{fig:sfp_clause_assignment_unit}.
	}
	\label{fig:sfp_inconsistency_gadget_with_connections}
\end{figure}

Additionally, we introduce arcs between different gadgets that provides paths from $s^L$ to~$v^L$.
For this, we introduce one inconsistency gadget for every inconsistency and connect them as follows.
A clause assignment unit corresponding to an assignment~$L = T_{Y(C_i)}$ of clause~$C_i$ gets an $s^L$-$v^L$-path~$P^L$ that passes through every inconsistency gadget for an inconsistency~$I = \{L, T_{Y(C_j)}\}$ containing~$L$.
In such a gadget, this path uses either the arc $v^I_1 v^I_2$ (if~$j > i$) or the arc~$v^I_3 v^I_4$ (if~$j < i$).
The path~$P^L$ collects all these arcs in arbitrary order by introducing further arcs between them.
The connecting arcs are indicated by dashed green arcs in \cref{fig:sfp_clause_assignment_unit,fig:sfp_inconsistency_gadget_with_connections}.

\subsubsection{Magnitudes and Parameters.}
Recall that we denote the number of clauses of~$\varphi = C_1 \vee \dots \vee C_m$ by~$m$ and the number of $x$- and $y$-variables by $n_x$ and~$n_y$, respectively (compare \cref{not:sigma2sat,not:sigma2sat3dnf}).
Also as before, let $q_i$ denote the number of occurrences of the $i$-th $x$-variable $x_i$ in the formula~$\varphi$.
Additionally, we denote the number of inconsistencies by~$n_I$.

The graph of the corresponding \SFP instance consists of one formula gadget, $n_x$~variable gadgets, and $n_I$~inconsistency gadgets.
The formula gadget consists of at most eight clause assignment units per clause.
Hence, we have $\bigO{m}$~clause assignment units.
Moreover, since we have at most one inconsistency gadget for every pair of clause assignment units, it holds that $n_I \in \bigO{m^2}$.

As shown in  \cref{lem:sfp_hardness-inconsistency_separation,lem:sfp_hardness-separating_sets_variable_gadget} we need one forbidden pair to separate $s^I$ and $t^I$ in the inconsistency gadget for~$I$ and $q_i + 1$~forbidden pairs to separate $s^i$ and $t^i$ in the variable gadget for~$x_i$.
Since the formula gadget itself is not connected, we do not need additional forbidden pairs for it.
Thus, for the typification framework, we choose
\begin{align}
	p = \max_{i = 1, \dots, n_x} q_i + 2 \label{eq:sfp:def-p}
\end{align}
and add $p$ parallel arcs from a source~$s$ to all input vertices of variable and inconsistency gadgets as well as to the formula gadget.
That is, we add all parallels of the form $s s^i$, $s s^I$, and $s s^0$.
Analogously, we add $p$~parallel arcs from every such output vertex to a target~$t$ resulting in parallels $t^i t$, $t^I t$, and $t^0 t$.
Furthermore, we add $p + 1$~parallel arcs from every input vertex of such a gadget to the output vertices of all other gadgets.
In total, we introduce
\begin{align*}
	2 p (n_x + n_I + 1) + (p + 1) (n_x + n_I + 1) (n_x + n_I) \in \bigO{m^5}
\end{align*}
arcs for the typification, where the asymptotic complexity~$\bigO{m^5}$ follows since $n_I \in \bigO{m^2}$ and since both, $n_x$ and $p$, are bounded by the number $3 \cdot m$ of literals in~$\varphi$.
As explained in the typification section we need
\begin{align}
	k_0 = p^2 (n_x + n_I + 1) (n_x + n_I) \label{eq:sfp:def-k0}
\end{align}
pairs to separate $s$ and $t$ in the graph that only consists of arcs introduced for typification.
We show in the analysis section below that we can separate the graph~$G$ by
\begin{align}
k = k_0 + n_I + \sum_{i = 1}^{n_x} (q_i + 1) \label{eq:sfp:def-k}
\end{align}
forbidden pairs if and only if the Boolean formula $\varphi$ has an $x$-variable assignment such that $\varphi$ evaluates to \true for every $y$-variable assignment.

\subsection*{Analysis}

So far, given a Boolean formula~$\varphi$, we have constructed an \SFP instance $G$ with source~$s$ and target~$t$ and specified the number~$k$ of forbidden pairs.
In the following, we use this to give a proof for \cref{thm:sfp_sigma2p-complete}, divided into \cref{lem:sfp_hardness-graph_acyclic,lem:sfp_hardness-g_poly_size,lem:sfp_hardness-lower_bound_separating_pairs,lem:sfp_hardness-satisfiable-k_pairs_sufficient,lem:sfp_hardness-not_satisiable-k_pairs_insufficient}.

\begin{lemma} \label{lem:sfp_hardness-graph_acyclic}
	The graph~$G$ that is constructed as described above is acyclic.
\end{lemma}
\begin{proof}
	As a graph is acyclic if and only if it exists a topological ordering, we prove the claim by specifying such a topological ordering for~$G$.
	However, we do not explicitly map every vertex to a natural number.
	Instead, we describe a procedure how to obtain the order of the vertices.
	The reason is that we have to insert some vertices in between others multiple times.
	This would make a formal definition of this mapping quite technical.

	In a first step, we enumerate all vertices in the formula gadget together with the interior vertices from inconsistency gadgets.
	Here, the ``interior vertices'' of an inconsistency gadget for an inconsistency~$I$ are the vertices $v^I_1, \dots, v^I_4$.
	Note that each arc of the form $v^I_1 v^I_2$ and~$v^I_3 v^I_4$ is contained in  an $s^L$-$v^L$-path~$P^L$ of some clause assignment unit.
	We start to enumerate the vertices in clause assignment units corresponding to the first clause~$C_1$.
	There, we first enumerate the paths~$P^L$ of assignments~$L$ for~$C_1$ followed by the output vertices~$t^L$ of the corresponding gadgets.
	Afterwards, we proceed in the same way with the subsequent clauses.
	This procedure is visualized in \cref{fig:sfp_topo_sorting_formula_gadget}.

	\begin{figure}
		\centering
		\newcommand{\drawAssignmentUnit}[3]{%
			\draw[contourLine] (#1, #2) -- ++(0, .5) -- ++(3, 0) -- ++(0, -1) -- ++(-3, 0) -- cycle;
			
			\node[smallVertex, fill=white] (s-#3) at (#1, #2) {};
			\node[smallVertex, fill=white] (v-#3) at ($(s-#3) + (1.75, 0)$) {};
			\node[smallVertex, fill=white] (t-#3) at ($(s-#3) + (3, 0)$) {};
			
			\draw[snakeArc] (s-#3) -- (v-#3);
			
			\coordinate (c11-#3) at ($(v-#3) + (.3, .35)$);
			\coordinate (c12-#3) at ($(t-#3) + (-.3, .35)$);
			\coordinate (c21-#3) at ($(v-#3) + (.3, .2)$);
			\coordinate (c22-#3) at ($(t-#3) + (-.3, .2)$);
			
			\draw[stdArc, color=Gray!40] (v-#3) -- (c11-#3) -- (c12-#3) -- (t-#3);
			\draw[stdArc, color=Gray!40] (v-#3) -- (c21-#3) -- (c22-#3) -- (t-#3);
			\draw[stdArc, color=Gray!40] (v-#3) edge [bend right] (t-#3);
		}
		
		\begin{tikzpicture}[scale=.8]
			\draw[contourLine] (0, 2.5) -- (13.5, 2.5) -- (13.5, -2.5) -- (0, -2.5) -- cycle;
			
			\node[stdVertex, fill=white] (s) at (0, 0) [label=left: $s^0$] {};
			\node[stdVertex, fill=white] (t) at (13.5, 0) [label=right: $t^0$] {};
			
			\node at (1, 0) {$\cdots$};
			
			\drawAssignmentUnit{2}{1}{1-1}
			\drawAssignmentUnit{2}{-1}{1-2}
			
			\drawAssignmentUnit{8}{1.8}{2-1}
			\drawAssignmentUnit{8}{.6}{2-2}
			\drawAssignmentUnit{8}{-.6}{2-3}
			\drawAssignmentUnit{8}{-1.8}{2-4}
			
			\node at (12.5, 0) {$\cdots$};
			
			\foreach \j in {1,2}
			{
				\foreach \i in {1,...,4}
				{
					\draw[stdArc] (t-1-\j) edge (s-2-\i);
				}
			}
		
			\begin{scope}[on background layer]
				\draw[pathMarker1, ->, line cap=round, line join=round]
					(1.5, .5) -- (s-1-1.center) -- (v-1-1.center) -- (s-1-2.center) -- (v-1-2.center) -- (t-1-1.center) -- (t-1-2.center) -- (s-2-1.center) -- (v-2-1.center) -- (s-2-2.center) -- (v-2-2.center) -- (s-2-3.center) -- (v-2-3.center) -- (s-2-4.center) -- (v-2-4.center) -- (t-2-1.center) -- (t-2-4.center) -- (12.5, -.5);
			\end{scope}
		\end{tikzpicture}		\caption[Topological ordering of the formula gadget.]{%
			A schematic representation how to enumerate the vertices in a formula gadget for a topological ordering.
			The $s^L$-$v^L$-paths within the clause assignment units are drawn as wavy lines.
			The paths and lines that might connect $v^L$ and~$t^L$ are only indicated.
			This visualizes the first step in the proof of \cref{lem:sfp_hardness-graph_acyclic}.
		}
		\label{fig:sfp_topo_sorting_formula_gadget}
	\end{figure}
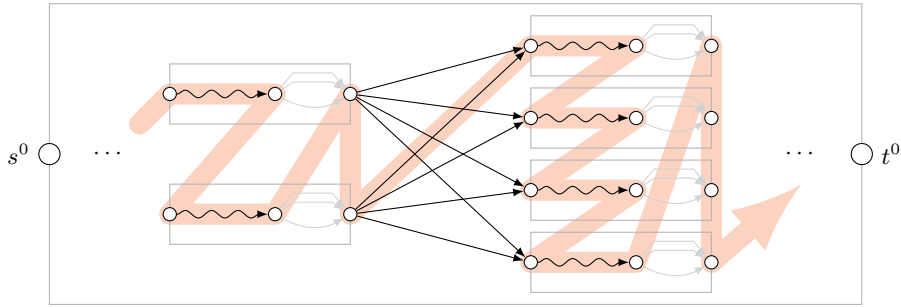

	By enumerating the formula gadget that way, for $i < j$ every vertex corresponding to a clause~$C_i$ gets a lower number than every vertex corresponding to clause~$C_j$.
	This holds in particular for the vertices on the $s^L$-$v^L$-paths~$P^L$ in the clause assignment units.
	For every inconsistency~$I = \{L = T_{Y(C_i)}, L' = T_{Y(C_j)}\}$ with $i < j$, the path $P^L$ uses the arc $v^I_1 v^I_2$ and the path $P^{L'}$ uses the arc $v^I_3 v^I_4$ within the corresponding inconsistency gadget.
	Thus, the partial topological ordering defined up to this point is not only consistent with all arcs of the formula gadget and the arcs in between formula and inconsistency gadgets, but also within all these inconsistency gadgets.

	It remains to prove that we can extend this partial ordering to the variable gadgets and the missing in- and output vertices.
	The latter are, however, no problem as we can put all input vertices at the beginning and all output vertices at the end, directly after~$s$ or before~$t$ (except for the in- and outputs of clause assignment units that already are assigned a number in the first step).

	Thus, in a second step, we have to assign numbers to the vertices of the variable gadgets.
	Such a variable gadget corresponding to a variable~$x_i$ consists of $q_i + 1$ many $s^i$-$t^i$-paths.
	With the exception of the additional path, every path corresponds to one occurrence of this variable.
	Let us consider the $j$-th occurrence and let~$C$ be the corresponding clause.
	Depending on whether $x_i$ occurs negated or not, we have restrictions either for the values of $v^i_{j, 5}$ and~$v^i_{j, 7}$ or for the values of $v^i_{j, 1}$ and $v^i_{j, 3}$, respectively (as those have arcs to vertices in clause assignment units that are already assigned a number).
	In particular, we only have restrictions on the ``left half'' or on the ``right half'' of the path but not on both.
	In the case the $j$-th occurrence is not negated, we assign the vertices $v^i_{j, 1}, \dots, v^i_{j, 3}$ increasing values that we insert in between the highest number of a vertex~$v^L$ and the lowest number of a vertex~$t^L$ in the topological ordering for every assignment~$L$ of clause~$C$.
	Note that we have enumerated these vertices in the first phase, such that the highest number of a vertex~$v^L$ is in fact smaller than the lowest number of a vertex~$t^L$ for all assignments~$L$ of clause~$C$.

	As all paths in the variable gadget are only connected to the additional path $(s^i, v^i_{0, 1},  v^i_{0, 2},  v^i_{0, 3},  v^i_{0, 5},  v^i_{0, 6},  v^i_{0, 7}, t^i)$, we can extend the partial topological ordering within every variable gadget.
	Therefore, we can assign $v^i_{0, 1}$, $v^i_{0, 2}$, and $v^i_{0, 3}$ values that are smaller and $v^i_{0, 5}$, $v^i_{0, 6}$, and $v^i_{0, 7}$ values that are larger than any values of vertices within the variable gadget.
	Thereafter, we can insert the ``missing half'' of paths accordingly.
	\qed
\end{proof}

\begin{lemma} \label{lem:sfp_hardness-g_poly_size}
	The graph~$G$ corresponding to the formula~$\varphi$ is of polynomial size and it can be constructed in polynomial time, both with respect to the size of~$\varphi$.
\end{lemma}
\begin{proof}
	For a given instance~$\varphi(x, y) = C_1 \vee \dots \vee C_m$ with $n_x$~many $x$-variables, let $G$ be the graph as described in this section.
	Its size is polynomial in the size of~$\varphi(x, y)$ as it contains $\bigO{m}$ clause assignment units, $n_x$~variable gadgets, and $\bigO{m^2}$ inconsistency gadgets.
	The size of the clause assignment and inconsistency gadgets is constant and the size of a variable gadget is linear in the number of occurrences of the corresponding $x$-variable.
	We add $\bigO{m^5}$ arcs for the typification and to see that also only polynomially many arcs connect different gadgets we can associate these to at least one of the two corresponding gadgets.
	Every inconsistency gadget is connected by exactly four inter-gadget arcs and every clause assignment unit is connected by either two, four, or six inter-gadget arcs.
	All the arcs connecting a variable gadget to other gadgets have the other endpoint in a clause assignment unit and are thus already considered.
	Hence, the number of inter-gadget arcs is polynomially bounded.
	Moreover, we can also construct the graph~$G$ from the formula~$\varphi$ in polynomial time.
	\qed
\end{proof}

\pagebreak[2]

\begin{lemma} \label{lem:sfp_hardness-lower_bound_separating_pairs}
	At least $k$ forbidden pairs are required to separate $s$ and~$t$ in~$G$, where $k$ is defined as in \cref{eq:sfp:def-k} on \cpageref{eq:sfp:def-k}.
\end{lemma}
\begin{proof}
	This follows from the typification construction.
	Every set of forbidden pairs~$\cA$ has to contain
	\begin{itemize}[nolistsep]
		\item the $k_0$ pairs to separate the graph that consists only of typification arcs,
		\item the $n_I$ pairs to separate all inconsistency gadgets (see \cref{lem:sfp_hardness-inconsistency_separation}), and
		\item for every $x_i$ either $\cA_i$ or $\overline{\cA}_i$ (see \cref{lem:sfp_hardness-separating_sets_variable_gadget}). \qed
	\end{itemize}
\end{proof}

\begin{lemma} \label{lem:sfp_hardness-satisfiable-k_pairs_sufficient}
	If the \SigmaTwoSat instance~$\varphi$ is satisfiable, we can separate $s$ and~$t$ in~$G$ by $k$ forbidden pairs, where $k$ is defined as in \cref{eq:sfp:def-k} on \cpageref{eq:sfp:def-k}.
\end{lemma}
\begin{proof}
	If $\varphi$ is satisfiable, there is an $x$-variable assignment~$T_X$ such that~$\varphi$ evaluates to \true no matter which values are assigned to the $y$-variables.
	We define a set of forbidden pairs depending on $T_X$ as follows.

	First, it contains the $k_0$ forbidden pairs that separate the graph consisting of typification arcs.
	Second, it contains the $n_I$ pairs that separate all inconsistency gadgets, compare \cref{lem:sfp_hardness-inconsistency_separation}.
	And finally, we choose a separating set for every $x$-variable~$x_i$.
	If $T_X(x_i) = 1$, we use the separating set~$\cA_i$.
	Otherwise, we use~$\overline{\cA_i}$.
	See \cref{lem:sfp_hardness-separating_sets_variable_gadget} for more information on these two sets.

	By \cref{eq:sfp:def-k,lem:sfp_hardness-inconsistency_separation,lem:sfp_hardness-separating_sets_variable_gadget} we have chosen $k$ forbidden pairs.
	Moreover, by the typification construction, all paths that do not use any gadget and those whose first and last gadgets are not the same contain a forbidden pair.

	It remains to prove that every path that enters a gadget via a direct arc from~$s$ and leaves this gadget via a direct arc to~$t$ completely contains at least one forbidden pair.
	We consider the different gadgets.

	First, consider an inconsistency~$I$ and the corresponding inconsistency gadget.
	Every path entering this gadget via~$s s^I$ must also use the arc~$v^I_1 v^I_2$.
	Analogous, every path leaving this gadget via~$t^I t$ must also use the arc~$v^I_3 v^I_4$.
	Thus, every path entering and leaving this gadget via input~$s^I$ and output~$t^I$ contains the forbidden pair $\{v^I_1 v^I_2, v^I_3 v^I_4\}$ that we have chosen.

	Next, consider the variable gadget for a variable~$x_i$.
	Similarly to the inconsistency gadget, a path entering the gadget via~$s^i$ can leave the gadget at the earliest at some vertex~$v^i_{j, 3}$ and, thus, it has to contain the pair $\{v^i_{j, 1} v^i_{j, 2}, v^i_{j, 2} v^i_{j, 3}\}$.
	Analogous, a path leaving the gadget via~$t^i$ must enter the gadget at the latest at some vertex $v^i_{j', 5}$ and, thus, it has to contain the pair $\{v^i_{j', 5} v^i_{j', 6}, v^i_{j', 6} v^i_{j', 7}\}$.
	At least one of these two pairs is contained in the set of forbidden pairs we have chosen.

	Finally, let us consider the formula gadget and let~$P$ be a path that enters the gadget via~$s^0$ and leaves it finally via~$t^0$.
	The path~$P$ passes through multiple inconsistency gadgets.
	If it uses more than one arc from one of them, it directly contains the forbidden pair chosen in this inconsistency gadget.
	Thus, we can assume that~$P$ uses at most one arc from every inconsistency gadget.

	If the path~$P$ leaves some clause assignment unit for a clause~$C$ via an arc to a variable gadget, it has to leave this variable gadget via an arc to the vertex~$t^L$ of a clause assignment unit that also corresponds to clause~$C$.
	Thus, for every clause, this path enters exactly one clause assignment unit via its input~$s^L$ and uses the $s^L$-$v^L$-path~$P^L$ therein.
	Consequently, if~$P$ passes through clause assignment units corresponding to inconsistent assignments $L$ and~$L'$, it contains the forbidden pair contained in the inconsistency gadget for $I = \{L, L'\}$.

	Therefore, we can assume that $P$ enters only clause assignment units corresponding to consistent assignments.
	This allows us to define a global $y$-variable assignment~$T_Y$ by combining the local clause assignments related to the clause assignment units that~$P$ enters via~$s^L$.
	As~$\varphi$ is satisfiable and $T_X$ is chosen appropriately we have that $\varphi$ evaluates to \true with $T = (T_X, T_Y)$.
	In particular, there is at least one clause~$C$ that is fulfilled.
	Let us consider the clause assignment unit associated to~$C$ that~$P$ enters via~$s^L$.
	As this clause is fulfilled, all $y$-literals are \true and, thus, the arc~$v^L t^L$ is not present.
	Hence, the path~$P$ has to pass through an $x$-variable gadget.
	However, as also this $x$-literal in~$C$ is \true, by the construction of the graph and the choice of the forbidden pairs, $P$ has to use a forbidden pair in this variable gadget.

	In all possible cases, the path~$P$ contains a forbidden pair.
	Thus, $s$ and~$t$ can be separated in~$G$ by $k$ forbidden pairs.
	\qed
\end{proof}

\begin{lemma} \label{lem:sfp_hardness-not_satisiable-k_pairs_insufficient}
	If the \SigmaTwoSat instance~$\varphi$ is not satisfiable, we cannot separate $s$ and~$t$ in~$G$ by $k$ forbidden pairs, where $k$ is defined as in \cref{eq:sfp:def-k} on \cpageref{eq:sfp:def-k}.
\end{lemma}
\begin{proof}
	Suppose for the sake of a contradiction that we can separate $s$ and~$t$ in~$G$ by $k$ forbidden pairs.

	By \cref{lem:sfp_hardness-lower_bound_separating_pairs} we need at least $k$ forbidden pairs.
	By the typification construction and by \cref{lem:sfp_hardness-inconsistency_separation,lem:sfp_hardness-separating_sets_variable_gadget} we have to choose the forbidden pairs from $\cA_i$ or~$\overline{\cA_i}$ in a variable gadget corresponding to variable~$x_i$.

	We define an $x$-variable assignment~$T_X$ based on this set of forbidden pairs.
	A variable~$x_i$ is set to \true if we have chosen $\cA_i$ to separate its variable gadget.
	Otherwise, if we have chosen $\overline{\cA_i}$, we set~$x_i$ to \false.

	As $\varphi$ is not satisfiable, there exists a $y$-variable assignment~$T_Y$ such that $\varphi$ evaluates to \false with $T = (T_X, T_Y)$.
	This $y$-variable assignment~$T_Y$ corresponds to exactly one clause assignment unit~$T_{Y(C)} = \restrict{T_Y}{Y(C)}$ for every clause~$C$.
	We now construct an $s$-$t$-path in~$G$ that does not contain a forbidden pair.
	This path starts with the arc~$s s^0$ and ends with~$t^0 t$.
	For every clause~$C$ it passes through the clause assignment unit corresponding to~$L = \restrict{T_Y}{Y(C)}$ where it first uses the $s^L$-$v^L$-path~$P^L$.
	If the clause contains a $y$-literal that is \false, the path continues along the arc~$v^L t^L$ that is present in this case.
	Otherwise, there is an $x$-literal that is not fulfilled.
	In this case, there exist a $v^L$-$t^L$-path through the corresponding variable gadget using two arcs that are not chosen as a forbidden pair (as this literal is \false).

	The path constructed this way does not contain a forbidden pair from a variable gadget.
	It does not contain a forbidden pair from an inconsistency gadget either as it only uses at most one arc from every inconsistency gadget.
	This is the case because it only uses consistent assignments for the clauses.
	And since the path does not contain a forbidden pair used to separate the graph consisting of only typification arcs, the path does not contain a forbidden pair at all.
	This contradicts our initial assumption and finishes the proof.
	\qed
\end{proof}

\section{Conclusion} \label{sec:conclustion}
Graph-theoretic problems based on paths are extensively studied.
In particular, this is the case for (arc- or vertex-) disjoint paths problems. 
If we relax this slightly and look for almost disjoint paths, we obtain a problem that is often required in applications but rarely researched from a theoretical point of view.
By allowing paths to have up to one arc in common, the almost disjoint paths problem is perhaps the most natural relaxation of the disjoint paths problem.
We have shown that this problem is already hard on directed acyclic graphs.
However, our dynamic program allows us to compute constantly many almost disjoint paths on this graph class in polynomial time.
Many other facets of \ADP's complexity are still open.
For example, it is unclear how the problem behaves for other graph classes (including undirected graphs), whether faster algorithms for the case of constantly many paths are possible, or how the problem changes when paths may have even more arcs in common.

Similarly, restricting \SFP to special graph classes and seeing how this affects its complexity is of interest.
Moreover, we would like to determine classes of graphs where \ADP and \SFP form a strongly dual pair, like for the class of graphs that have a cut with a single outgoing arc.
In such cases \SFP and \ADP are equally hard and it suffices to analyze the complexity of either one.

%
%
%
\bibliographystyle{splncs04}
\bibliography{mybibliography}

\end{document}